\documentclass[a4paper]{article}
\usepackage[utf8]{inputenc}
\usepackage[english]{babel}
\usepackage{subcaption}
\usepackage{enumerate}
\usepackage{amsfonts}
\usepackage{amsmath,amssymb}
\usepackage[makeroom]{cancel}
\usepackage{stmaryrd}
\usepackage{tikz-cd} 
\usepackage{graphicx}
\usepackage{amsthm}
\usepackage{wrapfig}
\newtheorem{theorem}{Theorem}
\newtheorem{proposition}[theorem]{Proposition}

\newtheorem{corollary}{Corollary}[theorem]
\newtheorem{lemme}[theorem]{Lemma}
\theoremstyle{definition}
\newtheorem{definition}[theorem]{Definition}
\newtheorem{th-definition}[theorem]{Theorem-Definition}
\newtheorem{cor-definition}[theorem]{Corollary-Definition}
\numberwithin{theorem}{section}

\numberwithin{equation}{section}
\begin{document}
\title{\center{\textbf{\large{ON THE REALISABILITY OF SMALE ORDERS}}}}
\author{\normalsize{I. IAKOVOGLOU}}
\date{}
\maketitle
\begin{abstract}
In this paper we're considering the question of whether or not a partial order on a finite set is realisable as a Smale order of a structurally stable diffeomorphism or flow acting on a closed manifold. In the following pages we classify the orders that are realisable by 1) an $\Omega$-stable diffeomorphism acting on a closed surface 2) an Anosov flow on a closed 3-manifold 3) a stable diffeomorphism with trivial attractors and repellers acting on a closed surface
\end{abstract}

\section{Introduction and formulation of the results}
\par{Is every partial order on a finite set realisable as a Smale order associated to a structurally stable diffeomorphism or flow acting on a closed manifold? The motivation for this question comes from a similar question asked by S.Smale in \cite{Smale} (problem 6.6) for $\Omega$-stable systems, that to our knowledge remained until now unanswered. Our goal in this article, is to answer Smale's question and to generalise it for structurally stable systems in dimension 2 and 3. By doing so, we hope to make more clear the abundance of structurally stable systems in dimension 2 with a prescribed Smale order and at the same time the topological restrictions that the dimension imposes on the Smale orders.} 
\par{In 1967, during a period when the problem of classification of $C^1$ dynamical systems was undergoing wide  research, S.Smale defined in \cite{Smale} a notion of order between hyperbolic \emph{basic pieces}, that was later known as \emph{Smale's order}. More specifically, consider $M$ a closed manifold and $f$ a diffeomorphism acting on $M$ satisfying the \emph{Axiom A}. The non-wandering set of such a diffeomorphism can be represented as a finite union of disjoint closed invariant sets $\Omega_1, . . . ,\Omega_k$, called basic pieces, each of which contains a dense orbit. From each $\Omega_i$ emanates a \emph{stable} and \emph{unstable lamination}, $W^s(\Omega_i)$ and $W^u(\Omega_i)$. We'll write $\Omega_i \rightarrow \Omega_j$ if $W^u(\Omega_i) \cap W^s(\Omega_j) \neq \emptyset$. If furthermore for every $\Omega_i,\Omega_j$ such that $\Omega_i \rightarrow \Omega_j$ there exist $p \in \Omega_i$ and $q \in \Omega_j$ two periodic points such that $W^u(p)$ and $W^s(q)$ intersect transversally at least at one point (this condition is known as \emph{Axiom B}), the relation $\rightarrow$ defines a partial order, which is called the Smale order of $f$.}
\par{Once he defined this notion of order, Smale naturally asked the question of whether all partial orders on finite sets can be realised as Smale orders of Axiom A and Axiom B systems. The answer to Smale's question is affirmative and the proof of this fact surprisingly enough relies essentially on the notion of \emph{DA map}. We therefore begin this article by proving the following (perhaps nowadays folklore) theorem:}
\begin{theorem}[Smale's question]\label{Smalequestion}
Any partial order on a finite set is realisable as a Smale order of an $\Omega$-stable diffeomorphism acting on a closed surface.
\end{theorem}
\par{A priori, during a period when non transitive Anosov flows were yet to be constructed, the question of realising any partial order as a Smale order of an Anosov flow was indeed a very difficult one. Today, thanks to the tools for constructing Anosov flows developed in \cite{Yu}, the question does no longer seem out of reach and gives a second (and not very difficult) example of hyperbolic system realising every possible Smale order.}
\begin{theorem}\label{Anosovcase}
Any partial order on a finite set is realisable as a Smale order of an Anosov flow acting on a closed 3-manifold.
\end{theorem}
\par{The main goal of this article is to address the question of whether every partial order can be realised as a Smale order of a structurally stable diffeomorphism in dimension 2. We soon distinguished two classes of stable diffeomorphisms, whose Smale orders had very different properties:}
\begin{enumerate}
\item[•]The stable diffeomorphisms with only trivial \emph{attractors} and \emph{repellers} (i.e. the maximal and minimal elements of the Smale order are periodic points)
\item[•]The stable diffeomorphisms with at least one non-trivial attractor or repeller 
\end{enumerate}  
\par{In the first case, we give a necessary and sufficient condition for a partial order to be realised, whereas the second case involves some more technical conditions that are going to be discussed in section 6. When all attractors and repellers are trivial, around every sink $p$, we can see the unstable manifolds of every saddle-type piece related to $p$ by the Smale order appear in a cyclic order. Furthermore, if $W^u(s_2)$ follows $W^u(s_1)$ in this cyclic order, then between the two lies the unstable manifold of a repelling point $a$ related to both $s_2$ and $s_1$. The existence of such a cycle translates to a necessary condition on the Smale order that is not always guaranteed, and will be called the \emph{connectivity condition}, which is going to be defined in section 5 in more detail. The connectivity condition narrows down the possible orders arising from stable diffeomorphisms with trivial attractors and repellers in dimension 2, but it is the only obstacle to the realisability of a partial order.} 
\begin{theorem}\label{Stablecase}
A partial order on a finite set is realisable as a Smale order of a structurally stable diffeomorphism with trivial attractors and repellers acting on a closed (not necessarily connected) surface if and only if the order satisfies the connectivity condition.
\end{theorem}
\par{In the proof of this theorem we show how to algorithmically construct for any partial order satisfying the connectivity condition an infinite number of distinct conjugation classes of stable diffeomorphisms realising this order. Generalising this theorem to diffeomorphisms with non-trivial attractors is indeed a difficult task. We will speak briefly about this in section 6, but the question still remains open. Before going on with the proofs of the previous theorems, I would like to thank Christian Bonatti for his most useful guidance and intuition, without which this work wouldn't have been completed.}   

\section{Preliminaries}
\subsection{Hyperbolic dynamics}
Let us remind very briefly in this section some useful definitions and results in hyperbolic dynamics. For a more detailed introduction, we refer the reader to some well known references \cite{Smale} or \cite{Bowen} or \cite{Katok}. We fix for the rest of this section $M$ to be a compact smooth manifold and $f \in C^r$ ($r \geq 1$) a diffeomorphism of $M$.
 
\begin{definition}\label{definitionhyperbolicset}
We say that an invariant compact subset $\Lambda \subset M$ is \emph{hyperbolic} if there exist $E^u, E^s$ two continuous invariant subbundles of the restriction of $TM$ on $\Lambda$ such that for every $x\in \Lambda$
\begin{enumerate}
\item $E_{x}^u \oplus E_{x}^s = T_{x}M$
\item For every Riemannian metric $\Vert ~ \parallel $ on $M$, there exist $\lambda, \mu \in \mathbb{R}$ verifying $0<\lambda<1<\mu$ and $C>0$ uniform in $x$ such that $\Vert d_{x}f^n(v) \Vert \leq C\lambda^n \Vert v \Vert$ and $\Vert d_{x}f^{-n}(w) \Vert \leq C \mu^{-n} \Vert w \Vert$ for every $v \in E^s(x)$, $w \in E^u(x)$ and $n \geq 0$.   
\end{enumerate} 
Furthermore, we say that $\Lambda$ is \emph{locally maximal} when there exists $U$ an open neighbourhood of $\Lambda$ such that $\underset{n\in \mathbb{Z}}{\cap}f^n(U)=\Lambda$. 
\end{definition}
\begin{definition}
We say that $f$ satisfies the \emph{Axiom A} if its non-wandering set $\Omega(f)$ is hyperbolic and equal to the closure of the set of periodic points of $f$.  
\end{definition}
\begin{theorem}[Generalized stable manifold theorem]
If $\Lambda$ is a hyperbolic set for $f$ and $d$ the distance induced by some Riemannian metric on $M$, then for every $x$ there exists an injective $C^r$ immersion $J_x^s: E^s(x) \rightarrow M$ (where $E^{s}(x)$ was defined in \ref{definitionhyperbolicset}), whose image is called the stable manifold, such that
\begin{enumerate}
\item $W^s(x):= J_x^s(E^s(x))=\lbrace y \in M| d(f^n(y),f^n(x)) \underset{n \rightarrow +\infty}{\longrightarrow} 0 \rbrace$
\item $f(W^{s}(x))=W^{s}(f(x))$ and $d_0 J_x^s(E^s(x))=E^s(x)$
\item If $x,y \in \Lambda$ are close then $W^s(x)$, $W^s(y)$ are $C^1$ close on compact sets. 
\end{enumerate}
By considering instead of $f$ its inverse, we can define in the same way the unstable manifold $W^u(x)$ for every $x \in \Lambda$.
\end{theorem}
\begin{theorem}[Smale's decomposition] 
If $f$ satisfies the Axiom A, then $\Omega(f)$ can be decomposed in a finite union of closed invariant disjoint sets $\Omega_i$, which we'll call basic pieces, on which $f$ acts transitively. Furthermore, every $\Omega_i$ can be decomposed in a finite union of closed disjoint sets $X_{i,1},...,X_{i,n(i)}$ on which $f^{n(i)}$ is mixing and $f(X_{i,j})=X_{i,(j+1)}$ ($X_{i,(n(i)+1)}=X_{i,1}$)
\end{theorem}
\begin{theorem}[Phase theorem]
With the same hypothesis as above, the basic pieces of $f$ are locally maximal. Furthermore, if $\Omega_i$ is such a basic piece  $$W^{s}(\Omega_i):=\lbrace y \in M| d(f^n(y),\Omega_i) \underset{n \rightarrow +\infty}{\longrightarrow} 0 \rbrace = \cup_{x \in \Omega_i} W^s(x)= M$$Same for $W^u(\Omega_i)$.
\end{theorem}
\begin{definition}
A basic piece of $f$ will be called \emph{trivial} if it consists only of a periodic orbit. Otherwise it will be called \emph{non-trivial}.
\end{definition}
\begin{definition}
If $f$ satisfies the Axiom A and $(\Omega_i)_i$ are the basic pieces of $f$ defined as above, we'll write $\Omega_i \rightarrow \Omega_j$ if $W^u(\Omega_i) \cap W^s(\Omega_j) \neq \emptyset$. We'll say in this case, by abuse of language, that $\Omega_i $ is bigger than $\Omega_j$.
\end{definition}
\begin{definition}
If $f$ satisfies the Axiom A, we'll say that it also satisfies the \emph{Axiom B} if for every $\Omega_i, \Omega_j$ such that $\Omega_i \rightarrow \Omega_j$ there exist $p \in \Omega_i, q\in \Omega_j$ periodic points such that $W^u(p)$ and $ W^s(q)$ have at least one point of transverse intersection. Furthermore, we'll say that it satisfies the \emph{strong transversality condition} if for every $x,y \in \Omega(f)$, $W^s(x)$ and $W^u(y)$ intersect transversally at every point of intersection. 
\end{definition}
\begin{definition}
If $f$ satisfies the Axiom A, we say that it also satisfies the \emph{``no cycle" condition} if the relation $\rightarrow$ defined previously doesn't contain any sequence of the form $\Omega_{i_1} \rightarrow \Omega_{i_2}\rightarrow...\rightarrow \Omega_{i_n}=\Omega_{i_1}$.  
\end{definition}
\begin{th-definition}\label{AxiomBnocycles}
If $f$ satisfies the Axiom A and the Axiom B (or less generally the strong transversality condition), then it also satisfies the ``no cycle" condition. Even more, under the same hypothesis, $\rightarrow$ is a partial order and will be called the \emph{Smale order} of $f$. 
\end{th-definition}
\begin{definition}
If $f$ satisfies the Axiom A and the Axiom B, we'll say that a basic piece $\Omega$ of $f$ is a \emph{hyperbolic attractor} if it has an open neighbourhood $U$ such that $\cap_{n \geq 0}f^n(U) = \Omega$. By replacing $f$ by $f^{-1}$ we can define a \emph{hyperbolic repeller}. Any basic piece that is not a hyperbolic attractor or repeller will be called of \emph{saddle type}.
\end{definition}
\begin{proposition}
If $f$ satisfies the Axiom A and the strong transversality condition, then the minimal elements of the order $\rightarrow$ are hyperbolic attractors and the maximal elements of the order $\rightarrow$ are hyperbolic repellers. All the other elements are locally maximal saddle-type pieces.
\end{proposition}
For every Axiom A and Axiom B diffeomorphism $f$, there exists a smooth Lyapounov function $\psi:M \rightarrow \mathbb{R^+}$ such that $\psi(f(x))\leq \psi(x)$ for every $x\in M$, where the equality is obtained if and only if $x$ is non-wandering. It's easy to see that if two distinct basic pieces satisfy $\Omega_{j} \rightarrow \Omega_{i}$ then $\psi(\Omega_j)> \psi(\Omega_i)$. By a small perturbation we can assume that for any $x,y \in \Omega(f)$ we have $\psi(x)=\psi(y)$ if and only if $x$ and $y$ belong to the same basic piece. Therefore, every Lyapounov function with the previous properties defines a total order on the basic pieces compatible with the Smale order of $f$ and conversely for every total order compatible with the Smale order of $f$, there exists a Lyapounov function realising this order. A corollary of this fact is the following theorem.  
\begin{theorem}[Filtration theorem]\label{Filtrationtheorem}
Suppose that $f$ satisfies the Axiom A and B. For every total order on the set of basic pieces compatible with the Smale order of $f$, if we label $(\Omega_i)_{i \in \llbracket 1,...,n \rrbracket} $ in an increasing way, there exist compact submanifolds with boundary $M_0,...M_n$ such that
\begin{enumerate}
\item $M_0=\emptyset$ and $M_n=M$  
\item $M_{i-1}\subset int(f^{-1}(M_{i-1})) \subset f^{-1}(M_{i-1}) \subset f(M_i)\subset int(M_i)$
\item $\Omega_i=M_i-\overline{M_{i-1}} \cap \Omega(f)= \cap_{n\in \mathbb{Z}} f^n(M_i-\overline{M_{i-1}})$ 
\end{enumerate}
We will call $M_i-\overline{M_{i-1}}$ a filtrating plug.
\end{theorem}
Finally, before moving on to the next section, let us remind two well-known theorems in hyperbolic theory. 
\begin{definition}
We say that $f$ is \emph{$\Omega$-stable} if for any $g \in \text{\text{Diff}}^1(M)$ close enough to $f$ there exists $h: \Omega(f)\rightarrow \Omega(g)$ homeomorphism such that $ h \circ f = g \circ h$ 
\end{definition}
\begin{definition}
We say that $f$ is \emph{structurally stable} if for any $g \in \text{Diff}^1(M)$ close enough to $f$ there exists $h$ homeomorphism of $M$ close to the identity such that $ h \circ f = g \circ h$.
\end{definition}
\begin{theorem}[Palis, Smale]\label{Omegastabilitytheorem}
$f$ is $\Omega$-stable if and only if $f$ satisfies the Axiom A and the ``no-cycle" condition.
\end{theorem}
\begin{theorem}[Ma\~n\'e]\label{Manetheorem}
$f$ is structurally stable if and only if $f$ satisfies the Axiom A and the strong transversality condition.
\end{theorem}

\subsection{Structurally stable systems in dimension 2}
In the following section, following the lines of \cite{Bonatti}, we're going to go through some very important results concerning structurally stable systems in dimension 2. All the following results are proven in \cite{Bonatti}. We fix for the rest of this section, $S$ to be a compact oriented surface without boundary, $f \in C^r$ with $r \geq 1$ a structurally stable diffeomorphism preserving the orientation of $S$.

\begin{definition}
If $x\in \Omega(f)$, then any connected component of $W^s(x)- \lbrace x\rbrace$ (resp. $W^u(x)- \lbrace x\rbrace$) is called \emph{stable} (resp. \emph{unstable}) \emph{separatrix}.
\end{definition}
\begin{definition}
A hyperbolic set $K$ is called \emph{saturated} if for every $x,y\in K$, $W^s(x) \cap W^u(y) \subset K$.
\end{definition}
If one looks for a method for describing a saturated hyperbolic set $K$, one could try to isolate it from the others by considering a filtrating plug of $K$ (defined in theorem \ref{Filtrationtheorem}) and extend the dynamics on this manifold with boundary by gluing to its boundary punctured disks.
\begin{theorem}\label{Bigdomain}
If $K$ is a saturated hyperbolic set of $f$, for every $W$ filtrating plug of $K$ there exists an open set $U$ containing $W$ such that 
\begin{enumerate}
\item $U-W$ is a finite set of two by two disjoint punctured disks
\item For every $C$ compact subset of $U$ containing $K$, we have $K=\cap_{n\in \mathbb{Z}} f^n(C)$ 
\item $U$ is diffemorphic to a closed surface from which we removed a finite number of points $p_1,...,p_n$. If we denote this surface $\tilde{S}$ and $\tilde{f}$ the push-forward dynamics, then  $\tilde{f}$ can be continuously extended and the points $p_1,...,p_n$ act as attracting or repelling periodic points. 
\end{enumerate} 
\end{theorem}
\begin{theorem}\label{regularityoflaminations}
If $K$ is a saturated hyperbolic set that doesn't contain any hyperbolic attractor or repeller, then $W^s(K)$ and $W^u(K)$ are $C^{1,0}$ closed laminations in $U$ with empty interior (i.e. $C^1$ leaves with continuous tangent distributions).
\end{theorem} 
\begin{th-definition}\label{def-sboundary}
Let $K$ be a saturated hyperbolic set that doesn't contain any hyperbolic attractors or repellers. Consider the stable manifold $W^s(x)$ of a point $x \in K$ and a very small segment $l$ transverse to it at any point $y \in W^s(x)$. If one connected component of $l-\lbrace y \rbrace$ doesn't intersect $W^s(K)$ then we'll call $x$ an \emph{s-boundary}. If both connected components of $l-\lbrace y \rbrace$ don't intersect $W^s(K)$ then we'll call $x$ a \emph{double s-boundary}. The previous property doesn't depend on the choice of the point $y$ nor the segment $l$ provided that the segment is taken small. We define in the same way a \emph{u-boundary} and a \emph{double u-boundary}. 
\end{th-definition}
\begin{definition}
If we take $K$ any saturated hyperbolic set, a stable or unstable separatrix of a point $x \in K$ will be called \emph{free} if it doesn't intersect any point of $K$.
\end{definition}
\begin{theorem} \label{strcutretheoremforhyperbolicsaturatedsets}
Take $K$ any saturated hyperbolic set.
\begin{enumerate}\label{th-sbords}
\item If $x\in K$ is an s-boundary then every point in its orbit and every point in $W^s(x)\cap K$ is also an s-boundary
\item If $x \in K$ is periodic and an s-boundary then it has at least one free unstable separatrix
\item The number of periodic points that are s-boundaries is finite
\item If $x \in K$ has one free unstable separatrix, it is an s-boundary and a periodic point.
\item If $x \in K$ is an s-boundary it's contained in the stable manifold of a periodic point s-boundary in $K$. 
\item If $K$ doesn't contain any s-boundaries it's a hyperbolic repeller. If $K$ doesn't contain any s-boundaries or u-boundaries, then $f$ is Anosov. 
\item If $x\in K$ is a double s-boundary, then it's an isolated (in $K$) periodic point. Furthermore, the basic piece associated to this periodic point is smaller or not related to any other basic piece in $K$.
\end{enumerate}

\end{theorem}

Another way of describing the action of a saturated hyperbolic set that doesn't contain any hyperbolic attractors or repellers is to define its domain $\Delta(K)$. The reader can think of the domain of a hyperbolic set $K$ as the ``smallest" invariant submanifold of $S$ with boundary of finite topology containing $W^s(K)$ and $W^u(K)$. Since all the invariant manifolds of $K$ are contained in the domain, outside its domain the dynamics is not influenced by the existence of $K$. The advantage of the domain, compared to the filtrating neighbourhood described in \ref{Bigdomain} is that the filtrating neighbourhood depends too much on the choice of filtration, whereas the domain of a hyperbolic set is more canonical. The following theorem makes more explicit the previous statements
\begin{theorem}\label{domaintheorem}
For every saturated hyperbolic set $K$ that doesn't contain any hyperbolic attractors or repellers, there exists $\Delta(K)$ an invariant submanifold of $S$ (with boundary) of finite topology such that
\begin{enumerate} 
\item $\Delta(K)$ contains $W^s(K)$ and $W^u(K)$
\item The boundary of $\Delta(K)$ contains no circles, just lines. 
\item For every boundary line $l$ of $\Delta(K)$ there exists $n \in \mathbb{N}$ such that $f^n(l)=l$ and for every such $n$ the dynamics close to $l$ is a translation.
\item  $\Delta(K)$ is connected if and only if $W^s(K) \cup W^u(K)$ is connected
\item If $R$ is a polygon (topological disk) whose sides are $s_1,a_1,...,s_n,a_n$ with $s_i \subset W^s(K)$, $a_i \subset W^u(K)$ and all the $s_i$ are s-boundaries (or all the $a_i$ u-boundaries), then $R \subset \Delta(K)$ 
\item If an open invariant set $U$ contains $W^s(K)$, $W^u(K)$ and all the rectangles $R$ defined as in 5, then there exists a continuous embedding of $\Delta(K)$ in $U$
\item Take two hyperbolic saturated sets $K_1$, $K_2$ that satisfy the initial hypothesis. We can take $\Delta(K_1)$ and $\Delta(K_2)$ disjoint if and only if there are no two basic sets $k_1 \subset K_1$, $k_2 \subset K_2$ related by $\rightarrow$.
\end{enumerate}
\end{theorem}
Not every domain of a hyperbolic set is connected, not even for a transitive piece. Even if the domain of a piece is connected, its topology can be very complicated. What is the relation of the domain with the open neighbourhood $U$ of the theorem \ref{Bigdomain}? The answer is given by the following proposition.
\begin{proposition}\label{attractingrepelling}
Following the notations of \ref{Bigdomain}, there exists a continuous embedding of $\Delta(K)$ in $\tilde{S}$ such that every line of the boundary of $\Delta(K)$ goes from one periodic repelling point $p_i$ to some periodic attracting point $p_j$.   
\end{proposition} 
The previous proposition gives us a dynamical orientation of every line in the boundary of $\Delta(K)$. 
\begin{cor-definition}\label{extendeddomain}
By adding a finite number of points to the domain of a saturated hyperbolic piece, we can extend the domain, to a submanifold of $\tilde{S}$ with boundary and ``angle" points, which we'll call the \emph{extended domain}. The boundary of the extended domain consists of a finite number of circles.
\end{cor-definition}
\begin{definition}\label{lengthofextendeddomain}
Take a domain of a saturated hyperbolic set $K$ that doesn't contain any hyperbolic attractors or repellers and $\tilde{K}$ its associated extension. Every boundary component $C$ of $\tilde{K}$ contains a finite number of boundary arcs of $K$, which will be called \emph{the length of the boundary component} $C$.
\end{definition}
\begin{figure}[h!]
\begin{center}
\includegraphics[scale=0.3]{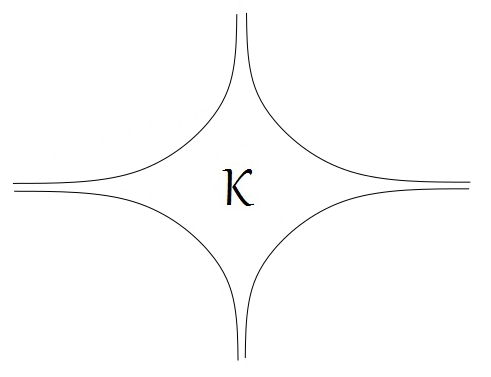}
\caption{Domain of a saturated hyperbolic piece}
\label{domain}
\end{center}
\end{figure}
In the figure \ref{domain} we can clearly see the unique circle forming the boundary of the extension of $\Delta(K)$. Notice that since every line of the boundary connects one attracting and one repelling point according to proposition \ref{attractingrepelling}, hence the cycles of lines in every domain's boundary have even length.  
\begin{theorem}\label{gluening}
Take a finite family of extended domains $\lbrace (\Delta(K_i),f_i)\rbrace$ as defined in the previous corollary. Orient the boundaries of the previous domains (take for instance the inward orientation of the boundary of every $\Delta(K_i)$) and denote by $\Gamma^{+}$ the lines of the boundary for which this orientation coincides with the dynamical orientation (defined after the proposition \ref{attractingrepelling}). Denote by $\Gamma^{-}$ all the other lines. $\Gamma^-$ and $\Gamma^{+}$ are invariant by $f$. If $\phi: \Gamma^{+} \rightarrow \Gamma^{-}$ is a bijection such that $\phi \circ f= f \circ \phi$, for every $\gamma \in \Gamma^{+}$ there exists a homeomorphism $h(\gamma):\gamma \rightarrow \phi(\gamma)$ that defines a gluing between $\gamma$ and $\phi(\gamma)$ and such that the union of the domains, glued along their boundaries by the $(h(\gamma))_{\gamma \in \Gamma^{+}}$ defines an oriented surface without boundary endowed with a structurally stable diffeomorphism $f$, whose restriction on the domain $\Delta(K_i)$ is exactly $f_i$. 
\end{theorem}
The above theorem states that gluing domains along their boundary constructs a stable diffeomorphism with trivial hyperbolic attractors and repellers. The converse is also true.
\begin{proposition}\label{unglue}
Every diffeomorphism $f$ acting on a closed surface with trivial hyperbolic attractors and repellers can be obtained as the result of gluing a set of domains along their boundaries, as described in the previous theorem.
\end{proposition}
\subsection{DA maps}
A useful method for constructing new hyperbolic sets from old ones is the DA map. Take any structurally stable diffeomorphism $f$ acting on a compact surface $S$, such that $f$ possesses a non-trivial basic piece $K$. We know from \cite{Bowen} that the action by $f$ on any such hyperbolic set $K$ can be described as a factor of a shift of finite type. Therefore, $K$ possesses many periodic points. Without any loss of generality, by considering $f^n$ for $n$ big enough, we can assume that $K$ contains many fixed points all of which have positive eigenvalues. Since $K$ is non-trivial, those fixed points are of type saddle. Take $p$ such a fixed point such that it doesn't have any free separatrix (such a point exists by theorem \ref{domaintheorem}). By taking an appropriate chart that sends $p$ to $0$, we can assume that $f$ acts on an open neighbourhood $U$ of $0 \subset \mathbb{R}^2$ and that $(0,1)$, $(1,0)$ are eigenvectors of $d_{0}f$ corresponding respectfully to the eigenvalues $\lambda_1 <1$, $\lambda_2 >1$. Consider a smooth function in $\phi:\mathbb{R} \rightarrow \mathbb{R}$ with the following properties:
\begin{enumerate}
\item $\phi(t)=\phi(-t)$
\item $\phi(x)=1$ if $x\in [-b,b]$, $\phi(x)=0$ for $x \in (-\infty,-a) \cup (a,+\infty)$, with $b<a$ positive real numbers
\item $\phi'(t)<0$ for every $t \in (-a,-b) \cup (b, a)$
\end{enumerate}
Consider now $\tilde{f}(x,y)=f(x,y)+((2-\lambda_1)\phi(kx)\phi(y)x, 0)$ where $k\in \mathbb{N}$ is taken big enough and $a$ small enough in the definition of $\phi$. 
\begin{figure}[h!]
\begin{center}
\includegraphics[scale=0.5]{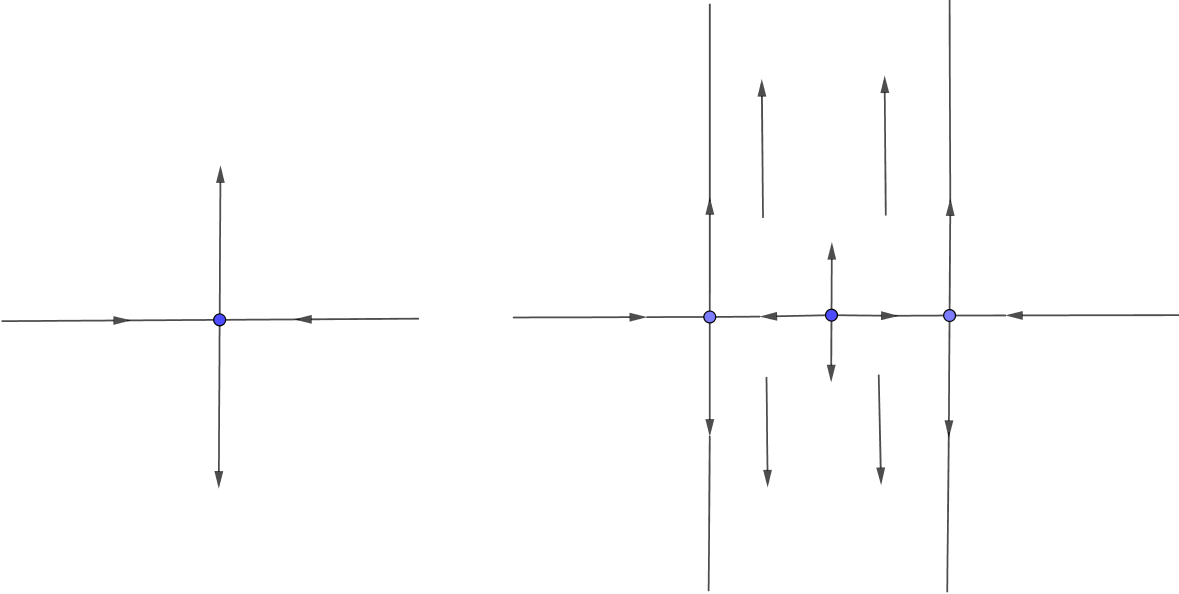}
\caption{Opening up the unstable separatrices}
\label{DA map}
\end{center}
\end{figure}
The point $0$ has become a repelling point for $\tilde{f}$ and $\tilde{f}$ has also two new fixed saddle points close to $0$. One can see this easily for a linear map and he could be convinced that if the perturbation's support is very close to $0$ then $f$ looks more and more like its linear part. We can now extend $\tilde{f}$ on $S$ (if the support of $\phi$ is taken small enough $f=\tilde{f}$ close to the boundary of $U$) to a diffeomorphism. The previous perturbation, as you can see in figure \ref{DA map}, opens up the unstable separatrices of the fixed point $p$. The fixed point and its unstable separatrices have been replaced by a band, whose boundary consists of two new saddle fixed points and their unstable separatrices and whose interior contains the basin of a new repelling point. The dynamics outside the band still remains the same. The hyperbolic piece $K$ is intact, except from its points on $W^u(p)$, that now have been divided each into two new points. Every stable manifold crossing $W^u(p)$ at the point $k$ now crosses the whole band that has been added and intersects its boundary and the two new points that $k$ has been divided into. 
\par{The previous operation adds one new repelling point to the ambient dynamics and is called DA repelling map. In a similar way, one can define the DA attracting map. Applying the DA map to a structurally stable system, gives a new structurally stable system. This is why, the DA maps are going to be of great importance to us and are going to be frequently used in the following pages. For a more detailed approach to DA maps, one can look at the very well analysed example in chapter 17.1 of \cite{Katok}.

\section{The answer to Smale's question, the $\Omega$-stable case }
\subsection*{Plugs and gluings of plugs}
In analogy with the notion of plug defined in \cite{Yu} we define the following:
\begin{definition}
Take an orientable surface with boundary $S$, $D(S)$ is a compact submanifold of $S$ and $f\in C^1(D(S),S)$ such that $f$ is injective, $f(D(S)) \cup f^{-1}(D(S))= S$ and close to the boundary of $S$ is conjugated smoothly to a translation (i.e. if $C$ is a boundary component of $S$, there exists a tubular neighbourhood $U$ of $C$ such that $f(U\cap D(S))\cap U \neq \emptyset$ and that $f$ is smoothly conjugated to a translation on the cylinder $U$). A couple $(S,f)$ satisfying the previous hypothesis will be called a \emph{plug}. Furthermore, if $C$ is a boundary component of $S$, we'll say that $C$ is an \emph{entrance boundary} if $D(S)$ contains a neighbourhood of $C$. Otherwise, if $C$ is contained in the image of $f$ we'll say that it's an \emph{exit boundary}. If a plug has only exit boundary components, we'll call it a \emph{repelling plug} and if it has only entrance boundary components we'll call it an \emph{attracting plug}.
\end{definition} 
\begin{figure}[h!]
\begin{center}
\includegraphics[scale=0.4]{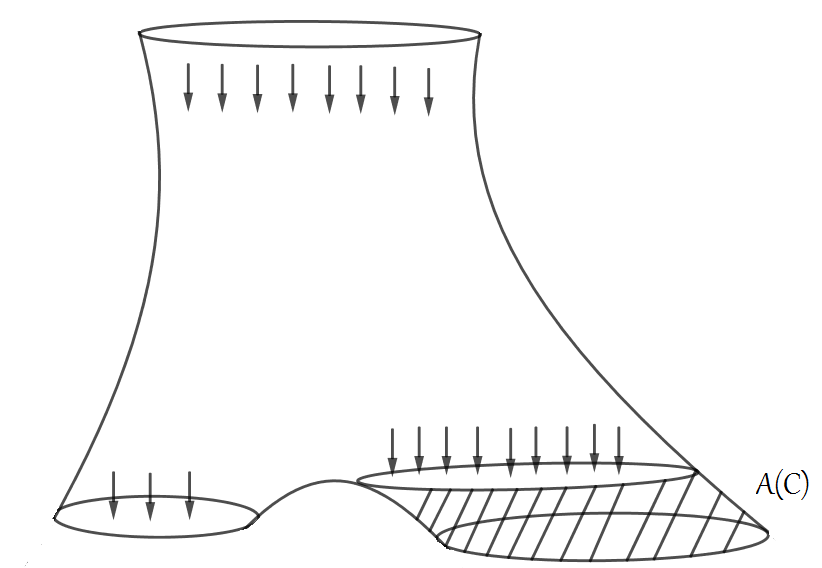}
\caption{An example of a plug}
\label{plug}
\end{center}
\end{figure}
In figure \ref{plug}, we give an example of a plug, with one entrance boundary and two exit boundaries. 
\begin{definition}
We'll say that a plug $(S,f)$ is \emph{hyperbolic} if the invariant set $\cap_{n\in \mathbb{Z}} f^n(S)$ is non-empty and hyperbolic. Furthermore, a hyperbolic plug will be called \emph{transitive} if the previous invariant set is a hyperbolic transitive set.  
\end{definition}
\begin{definition}
Take $(S,f)$ a hyperbolic plug and $K$ its hyperbolic invariant set defined as above. A boundary exit component $C$ of $S$ will be called \emph{useful} if $W^u(K) \cap C \neq \emptyset$. Similarly, a boundary entry component $C$ of $S$ will be called useful if $W^s(K) \cap C \neq \emptyset$.
\end{definition}
\begin{definition}
Take an exit boundary component $C$ of a hyperbolic plug $(S,f)$. The tubular closed neighbourhood $U$ of $C$, whose boundary is exactly $C$ and $f^{-1}(C)$ will be called the \emph{fundamental exit annulus} of $C$ and will be denoted by $\mathcal{A}(C)$ (see figure \ref{plug}). The torus defined by $U/f$ will be called the \emph{exit fundamental domain} associated to $C$ and will be denoted $\mathcal{T}(C)$. If $C$ is useful, then we can project $W^u(K)$ on the fundamental domain and this projection will be called \emph{exit lamination}. We define in the same way the \emph{entrance lamination}.    
\end{definition}
Take a diffeomorphism $f$ acting on a closed surface and a saturated hyperbolic set $K$. A filtrating plug of $K$ is a plug in the sense of the above definition. This is why, from now on we'll assume that the entrance boundary of any hyperbolic plug doesn't intersect $W^u(K)$ and intersects transversally $W^s(K)$. We make the analogous hypothesis for the exit boundary. In figure \ref{plug} the reader can find an example of a possible $W^s(K)$ lamination on the fundamental torus close to an entrance boundary component.
\begin{figure}[h!]
\begin{center}
\includegraphics[scale=1]{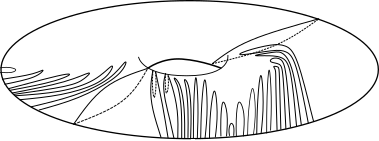}
\caption{}
\label{plug}
\end{center}
\end{figure}
\begin{definition}
Take two hyperbolic plugs (not necessarily distinct) $(S_1,f_1,K_1)$, $(S_2,f_2,K_2)$, $C_1$ an exit boundary component of $S_1$, $C_2$ an entrance boundary component of $S_2$. Take a diffeomorphism $\phi: \mathcal{A}(C_1) \rightarrow  \mathcal{A}(C_2)$ such that $\phi(C_1)=f_2(C_2)$ and $(\phi \circ f_1)_{f_1^{-1}(C_1)}=(f_2 \circ \phi)_{f_1^{-1}(C_1)}$ (see figure \ref{plugglue}). Consider now $S_1\sqcup S_2/ \phi$, which can be given a smooth structure of surface with boundary and on which $f_1$ can be extended continuously by $f_2$, by our previous hypothesis. We'll say that $\phi$ is a \emph{plug gluing} if the extension of $f_1$ by $f_2$ is differentiable in $S_1\sqcup S_2/ \phi$.  
\end{definition}
\begin{figure}[h!]
\begin{center}
\includegraphics[scale=0.4]{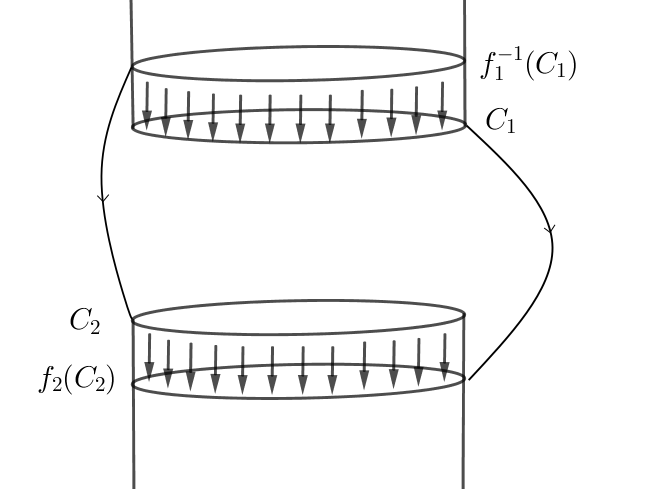}
\caption{}
\label{plugglue}
\end{center}
\end{figure}
\begin{definition}
Using the notations of the previous definition, if $\phi:\mathcal{A}(C_1) \rightarrow  \mathcal{A}(C_2)$ is a gluing then by definition we can define its projection $\overline{\phi}:\mathcal{T}(C_1) \rightarrow  \mathcal{T}(C_2)$. We'll call $\phi$ is a \emph{transverse gluing} if the image of the exit lamination of $C_1$ by $\overline{\phi}$ is a transverse lamination to the entrance lamination of $C_2$.
\end{definition}
\subsection*{Proof of theorem \ref{Smalequestion}}
We're going to prove that any partial order on a finite set can be realised as a Smale order of an Axiom A and Axiom B diffeomorphism $f$ acting on a closed surface. In the following lines, we'll prove something even stronger: that $f$ can be constructed in such way that for any two basic pieces of $f$ of saddle type $\Omega_i \rightarrow \Omega_j$, the unstable manifolds of $\Omega_i$ intersect transversally the stable manifolds of $\Omega_j$. Therefore, our construction will produce an Axiom A diffeomorphism satisfying the strong transversality condition everywhere, except at the levels of the hyperbolic attractors and repellers. We will see that this comes as no surprise, since satisfying the strong transversality condition at the levels of the attractors and repellers is not always possible. The proof of theorem \ref{Smalequestion} is based on the following ideas:
\begin{enumerate}
\item Given any natural numbers $n,m$ such that $m+n \neq 0$, there exists a transitive hyperbolic plug with $n$  useful entrance boundary components and $m$ useful exit boundary components.
\item A transverse gluing of two hyperbolic plugs defines a new hyperbolic plug
\item Take any two tori $(T_i)_{i=1,2}$ supporting each a $C^{1,0}$ lamination with empty interior $L_i$, along with an oriented simple closed curve $\gamma_i$ transverse to the lamination $L_i$. Consider for each lamination a leaf $l_i \in L_i$. There exists a diffeomorphism $\phi: T_1 \rightarrow T_2$ such that $\phi(L_1)$ is transverse to $L_2$, $\phi(l_1)$ intersects $l_2$ and $\phi(\gamma_1)=\gamma_2$ with $\phi$ respecting the orientation of the $\gamma_i$.  
\end{enumerate}
In the following lines, we're only going to prove steps 1 and 2. Step 3 is true for a generic diffeomorphism satisfying $\phi(\gamma_1)=\gamma_2$ and $\phi(l_1) \cap l_2 \neq \emptyset$. We leave the construction of such a diffeomorphism to the reader. 
\begin{proof}[Step 1]
Take an Anosov diffeomorphism $f$ on $\mathbb{T}^2$. By considering $f^N$ for $N \in \mathbb{N}$ big enough we can assume that $f$ has many fixed points. Choose one of them, say $p$ and apply to $p$ the DA repelling map. In this way, by our discussion on section 2.3, we get a new stable diffeomorphism $\tilde{f}$ on $\mathbb{T}^2$ for which $p$ is a repelling point. Remove a small open disk $D$ around $p$. Close to $p$ the dynamics is conjugated to a homothecy, therefore the couple $(\mathbb{T}^2-D, \tilde{f})$ is a plug. Furthermore, since the dynamics outside of the newly added band (see figure \ref{DA map}) is the same as before, the plug $(\mathbb{T}^2-D, \tilde{f})$ is a transitive hyperbolic plug with one useful entrance boundary component. By applying $m$ times the DA attracting map and $n$ times the DA repelling map, we get in the way we described before the prescribed number of useful entry and exit boundaries. 
\end{proof}
\begin{proof}[Step 2] 
\par{Suppose we glue two hyperbolic plugs $(S_1,f_1,K_1),(S_2,f_2,K_2)$ transversally along $\phi$, by identifying an exit fundamental annulus of $S_1$ with an entry fundamental annulus of $S_2$. By definition, we get a new diffeomorphism $F$ acting on $\widehat{S}:=S_1\sqcup S_2/ \phi$ that is equal to $f_1$ on $S_1$ and to $f_2$ on $S_2$, therefore the invariant manifolds of $K_1,K_2$ remained intact and even more $W^s(K_2)$ extended to $S_1$ and $W^u(K_1)$ to $S_2$. Let us now prove that $(\widehat{S}, F)$ is a hyperbolic plug.}
\par{Indeed it is easy to see that it is a plug. Let us now determine which is the maximal invariant set of $(\widehat{S}, F)$. The set $\cap_{n\geq 0} f^n( \widehat{S})$ consists of all points that remain for all their past iterations inside $\widehat{S}$ therefore $\cap_{n\geq 0} f^n( \widehat{S})= W^u(K_1 \cup K_2)$. Similarly, $\cap_{n \leq 0} f^n( \widehat{S})= W^s(K_1 \cup K_2)$. Hence, $$\cap_{n \in \mathbb{Z}} f^n( \widehat{S})= K_1 \sqcup K_2 \sqcup (W^s(K_2)\cap W^u(K_1))$$}
\par{Let us now prove that this set is hyperbolic. Indeed, every point of $W^s(K_2)\cap W^u(K_1)$ crosses at least once the region where we identified the two plugs. Therefore, if the gluing is transverse, the laminations $W^s(K_2)$ and $W^u(K_1)$ are everywhere transverse. Next, it is trivial to see that $K_1\sqcup K_2$ remained a hyperbolic set. Also, for every point $x \in W^s(K_2)\cap W^u(K_1)$ take $E^s_x= T_xW^s(K_2)\subset T_x\widehat{S}$ and $E^u_x= T_xW^u(K_1)$. By transversality, $E^s_x\oplus E^u_x = T_x\widehat{S}$. If $x$ is close to $K_1$, then $E^u_x$ is a repelling direction. Similarly, if $x$ is close to $K_2$, then $E^s_x$ is a contracting direction. The Lemma 1.2.5 in \cite{Bonatti} states that $K_1 \sqcup K_2 \sqcup (W^s(K_2)\cap W^u(K_1))$ is compact, therefore by changing the constant of hyperbolicity defined in definition \ref{definitionhyperbolicset} the direction $E^s_x$ is contracting for every point $x \in K_1 \sqcup K_2 \sqcup (W^s(K_2)\cap W^u(K_1))$. Similarly $E^u_x$ is repelling. The only thing remaining to show is that $E^s_x$ depends continuously on $x$. This is essentially based on the $\lambda$-lemma and a more detailed proof may be found in \cite{Robbin} (proven for $C^2$ diffeomorphisms) or \cite{Robinson} (proven in the general case).}
\end{proof}
\begin{proof}[Smale's question]
Take a partial order $(F,<)$ on a finite set and represent it by a non-oriented graph, as in the following example. For the sake of simplicity, we will not add in the graph the edges that originate from the transitivity property of the partial order.
\begin{figure}[h!]
\begin{center}
\includegraphics[scale=0.08]{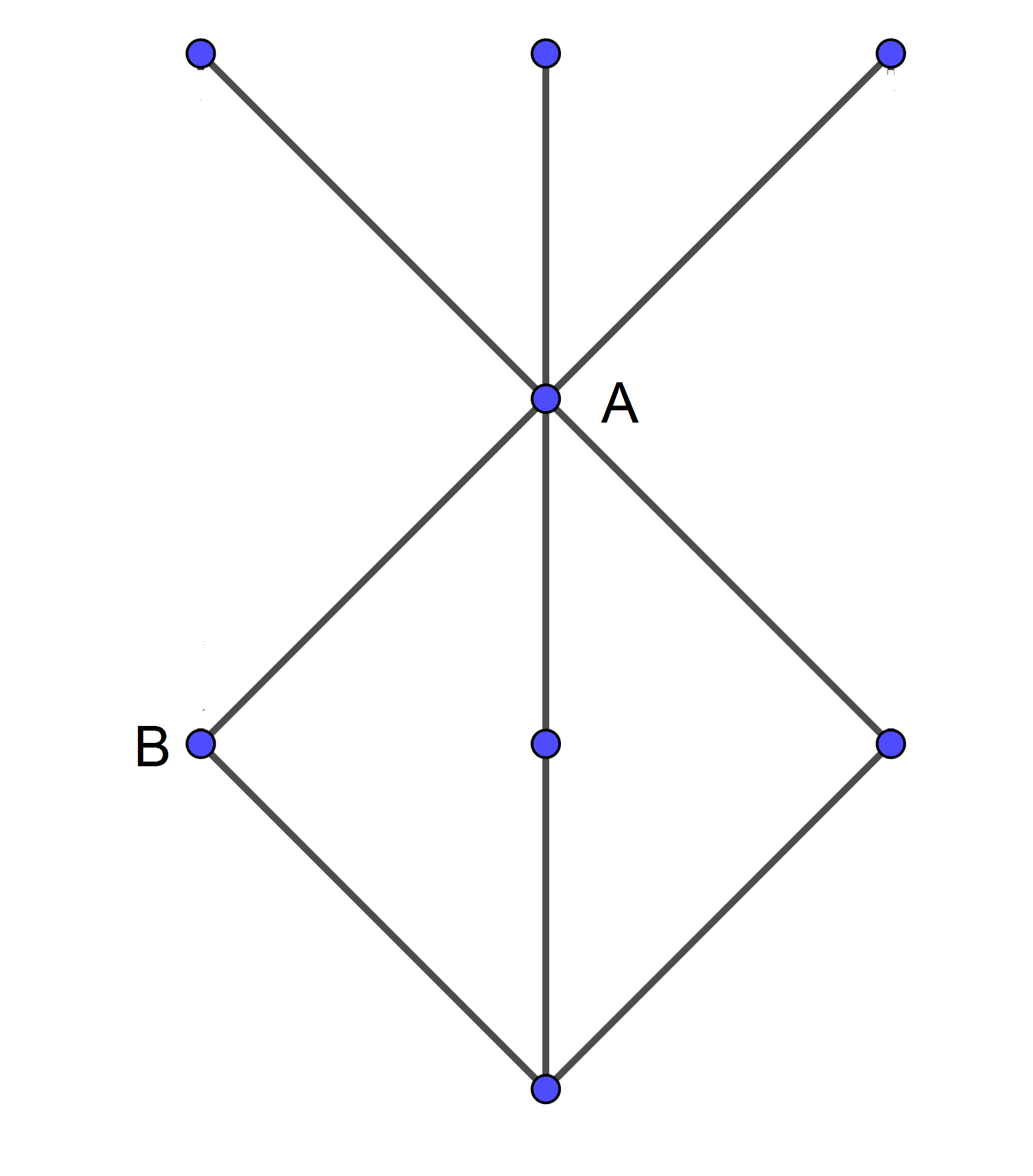}
\caption{The saddle $A$ has 3 children and 3 ancestors, the saddle $B$ has one child and 1 ancestor}
\label{order}
\end{center}
\end{figure}
\par{By convention, we'll assume that the highest level of points corresponds to hyperbolic repellers, the second highest to saddles of \emph{first generation} (i.e. hyperbolic pieces of type saddle whose stable manifolds don't intersect the unstable manifold of any other saddle-type piece), the third highest to \emph{second generation saddles} (i.e. saddles whose stable manifolds intersect the unstable manifold of hyperbolic repellers or first generation saddles) and so on. If the vertices $v>v'$ are connected by an edge, we'll say that $v$ is an \emph{ancestor} of $v'$ and $v'$ is a \emph{child} of $v$. In this article, the previous conventions will apply for every graph associated to a partial order.}
\par{For every hyperbolic piece in the graph of $(F,<)$ with $n$ ancestors and $m$ children, consider a hyperbolic transitive plug with $n$ useful entry components and $m$ useful exit components. We proved in step 1 that such plug exists. Next, let's glue together all the saddle-type pieces related by the partial order that we want to realise. Take for instance the example in figure \ref{order} and assume that we want to glue the plug associated to $A$ to the one of $B$. Consider an exit component $C$ of the plug corresponding to $A$ and an entry component $D$ of the plug corresponding to $B$. By step 2, we know that if we glue the two plugs along $\cal{A}(C)$ and $\cal{A}(D)$ transversally in a way such that $W^u(A)\cap W^s(B) \neq \emptyset$, then we construct a new hyperbolic plug, where $A \rightarrow B$. By step 3, since $W^u(A)$ and $W^s(B)$ are $C^{1,0}$ closed laminations of empty interior by the theorems \ref{regularityoflaminations} and \ref{strcutretheoremforhyperbolicsaturatedsets}, we obtain a diffeomorphism $\phi: \cal{A}(C) \rightarrow \cal{A}(D)$ that satisfies the previous properties. By perturbing eventually this diffeomorphism or the dynamics on each plug, we can assume that $\phi$ is a plug gluing. Therefore, by repeating this process for all saddle-type pieces, we can construct a big plug containing every saddle-type piece, whose Smale order is compatible with the partial order initially chosen, and inside which the transversality condition is satisfied. }
\par{Finally, we add the hyperbolic attractors and repellers. Now, a transverse gluing is not always possible, because the laminations of a hyperbolic non-trivial attractor or repeller have non-empty interior. Indeed, the theorem \ref{nontrivial} implies that a non-trivial attracting plug can never be transversally glued to a non-trivial repelling plug or a non-trivial plug of saddle type. This is why, instead of asking for transversality, we're going to ask that the axiom B be satisfied. }
\par{More specifically, assume that we want to glue the attracting non-trivial transitive hyperbolic plug associated to a minimal element $\omega$ of $(F,<)$ to the previously constructed saddle-type plug containing all saddle-type pieces. Let's name $A'$ the first plug and $B'$ the latter. Consider $C$ an entrance boundary component of $A'$. By construction of our transitive hyperbolic plugs in Step 1 and by our discussion in section 2.3, we know that the fundamental lamination of $C$ is of non-empty interior and contains two free stable separatrices coming from the two new fixed saddle points that appeared after the DA perturbation (see figure \ref{DA map}). From the other side, take an exit component $D$ of $B'$ belonging to a smaller plug in $B'$ (by construction $B'$ is a union of transitive plugs glued along their boundaries) associated to a basic piece $s$ bigger than $\omega$. It's easy to define a gluing of the two plugs, where the free separatrices of $\omega$ in $C$ are transverse at least at one point to some unstable manifold of $s$ in $D$. This gluing in most of the cases isn't going to satisfy the strong transversality condition, but it will satisfy the Axiom B. Therefore, after repeating this process for all entrance boundary components of $A'$ and for all attracting and repelling plugs, we obtain a system satisfying the Axiom A and B, so by theorem \ref{AxiomBnocycles}, it also satisfies the "no-cycle" condition and by theorem \ref{Omegastabilitytheorem} it defines an $\Omega$-stable system with the prescribed Smale order.}
\end{proof}

Let us make some remarks about our previous proof. We prove in the previous construction that the only obstruction for realising a general partial order on a finite set as a Smale order of a stable diffeomorphism is the gluing of the plugs at the levels of the hyperbolic attractors and repellers. But then, why can't we use trivial attracting and repelling plugs instead of non-trivial ones? 
\par{A trivial hyperbolic plug containing only a fixed attractive point has only one entrance component, therefore if two saddle type plugs $S_1,S_2$ must be glued to a trivial one, one of them, say $S_1$ must be glued to the trivial one and the other, $S_2$ must be glued to $S_1$. But what if $S_2$ and $S_1$ aren't related by the Smale order? Is such a gluing always possible? The answer to this question is negative. However, if the order that we want to realise is total, the previous problem does no longer exist and an adaptation of the previous method can be used to construct stable systems realising any total order. In order to produce stable systems that realise more complex partial orders, we need to be able to better control the exit and entry lamination of every plug. This is why we chose to present a different approach to this problem by using the notion of the domain of a hyperbolic piece more actively. }
\section{Proof of theorem \ref{Anosovcase}, the Anosov case}
Before moving on to the proof of the stable diffeomorphism case, let us prove that the arguments given in the previous section apply in the same exact way to the case of Anosov flows in dimension 3. Let us remind some basic definitions and results from \cite{Yu}. 
\begin{definition}
We call \emph{plug} any pair $(V,X)$, where $V$ is a 3 compact orientable manifold with boundary and $X$ a flow on $V$ transverse to its boundary. 
\end{definition}
Similarly to the previous section, the flow $X^t$ perhaps isn't defined for every $t \in \mathbb{R}$ and in the exact same way we can define the entrance (exit, useful) boundary and the hyperbolic (transitive, attracting, repelling) plug. In this section, we'll also make the same assumption as in the previous one, that every hyperbolic plug has been constructed with a filtrating neighbourhood of the corresponding hyperbolic set. Furthermore, we'll assume that our plugs have only useful boundary components.
\begin{definition}
If $(V,X,K)$ is a hyperbolic plug, then the restriction of $W^s(K)$ (resp. $W^u(K)$) to the entrance (resp. exit) boundary will be called the \emph{entry} (resp. \emph{exit}) \emph{lamination}. 
\end{definition}
\begin{definition}
We'll say that the entrance lamination is \emph{filling} if it can be extended to a foliation of the boundary. Same for the exit lamination.
\end{definition} 
\begin{definition}
We'll say that the entrance lamination is an \emph{MS-lamination} if
\begin{enumerate}
\item It has a finite number of compact leaves 
\item Every non-compact half leaf is asymptotic to a compact leaf
\item Each compact leaf can be oriented in such a way so that its holonomy is a contraction
\end{enumerate}
\end{definition}
\begin{definition}
Let $S$ be a closed surface and $L$ be a 1-dimensional lamination with finitely many compact leaves, and $C$ be a connected components of $S - L$. We call $C$ a \emph{strip} if it satisfies the two following properties:
\begin{enumerate}
\item[•]$C$ is homeomorphic to $\mathbb{R}^2$
\item[•] the accessible boundary of $C$ consists of exactly two non-compact leaves of $L$ which are asymptotic to each other at both ends
\end{enumerate}
Otherwise, we say that $C$ is an \emph{exceptional component} of $S- L$.
\end{definition}
\begin{definition}
Take $S$ an entrance boundary component of a plug $(V,X,K)$. We'll call the entrance lamination associated to $S$ an \emph{MS-filling lamination} if it's an MS-lamination whose complement in $S$ has no exceptional components. Same for the exit lamination.
\end{definition}
\begin{lemme}[Lemma 2.20 of \cite{Yu}]
If the entrance lamination of a hyperbolic plug $(V,X,K)$ is an MS-filling lamination then this is also the case for the exit lamination. Furthermore, an MS-filling lamination is filling, therefore an MS-filling lamination can only exist on the torus or the Klein bottle.
\end{lemme}
Also, as in the previous section one can show that the transverse gluing of two plugs is a plug.
\begin{proposition}[Proposition 1.1 of \cite{Yu}]\label{gluingplugs}
Let $(U,X,K)$ and $(V,Y,L)$ be two distinct hyperbolic plugs. Let $C$ be an exit component of $U$ and $D$ an entry component of $V$. Assume that there exists a diffeomorphism $\phi: C \rightarrow D$ such that $\phi(W^u(K))$ is transverse to the lamination $W^s(L)$. Let $Z$ be the vector field induced by $X$ and $Y$ on the manifold $W := U \sqcup V / \phi$. Then $(W,Z)$ is a hyperbolic plug. Furthermore, if the laminations are MS-filling and $W$ is closed, then $Z$ is an Anosov flow.
\end{proposition}
The idea behind the first part is exactly the same as the one we described in the previous section. For the second part of the statement, in order to prove that $Z$ is Anosov, we extend the laminations to two transverse foliations and we push by the flow those foliations in order to define a foliation on $W$. The hyperbolic behaviour of the new leaves is assured by the fact that every new leaf belongs to a strip, that accumulates in both directions to a compact leaf, whose holonomy is a contraction. 
\par{We now have an equivalent of step 2 of the previous section. In order to prove the theorem \ref{Anosovcase} we'll also need an equivalent of the notion of a DA map. By perturbing a little bit the flow and linearising in a $C^1$ way, we can define a DA attracting and repelling map, in the exact same way as in section 2.3 and prove the following proposition:}
\begin{proposition}[Proposition 7.1 of \cite{Yu}]\label{DAflow}
Let $(U,X,K)$ a hyperbolic plug of type saddle. Apply the DA repelling map to one periodic orbit $p$ (the orbit $p$ becomes repelling) with positive multipliers and no free separatrix. Let's name $X'$ the flow after the DA perturbation. We have the following results:
\begin{enumerate}
\item The couple $(U,X')$ is a hyperbolic plug.  
\item  Remove a small open tubular neighbourhood $S$ of the orbit $p$ and name $T$ the new boundary that appears in $U$. We have that $T$ is a torus and $(U-S,X')$ is a hyperbolic plug. Let's denote by $K'$ its associated hyperbolic set. 
\item $W^{s}(K') \cap \partial U$ is topologically equivalent to $W^s(K)\cap \partial U$. Same for $W^u(K')$
\item $W^s(K') \cap T$ is an MS-filling lamination diffeomorphic to the one drawn in figure \ref{laminationsda}
\end{enumerate}
\begin{figure}[h!]
\begin{center}
\includegraphics[scale=0.5]{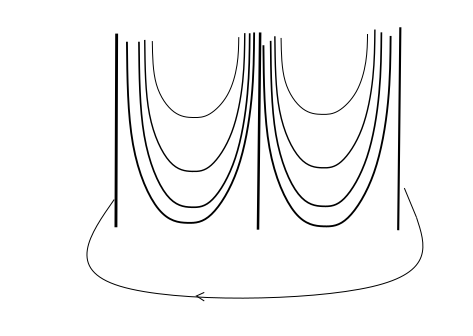}
\caption{}
\label{laminationsda}
\end{center}
\end{figure}
\end{proposition}  
By number 4 of the previous proposition, we can easily obtain the following
\begin{corollary}
Take $(U,X),(V,Y)$ two hyperbolic plugs. Apply the DA attracting map to a periodic orbit $p$ with no free separatrix and positive multipliers of $(U,X)$ and the DA repelling map to a periodic point $p'$ of $(V,Y)$ satisfying the same hypothesis. Remove a tubular neighbourhood of $p$ and $p'$, and name $T$ and $T'$ respectfully the two new boundaries in $U$ and $V$. We can always define a transverse gluing between the two plugs along $T$ and $T'$.
\end{corollary}
Finally, the proof of the theorem \ref{Anosovcase} is obtained by repeating the steps 1, 2, 3 defined in the previous section. The proof of step 1 relies on the fact that we can apply, similarly to the previous section, the DA attracting and repelling map to a transitive Anosov flow on a closed 3-manifold. Step 2 is replaced here  by proposition \ref{gluingplugs} and finally step 3 is replaced by proposition \ref{DAflow} and the above corollary.

\section{Proof of theorem \ref{Stablecase}, the stable case}
\subsection*{The connectivity condition}
Take a structurally stable orientation-preserving diffeomorphism $f$ acting on a closed oriented surface $S$. Suppose that $f$ has trivial hyperbolic attractors and repellers. Consider an attractive periodic point $p$ of $f$. By the theorem \ref{unglue}, we know that the surface can be seen as the result of a gluing of domains of hyperbolic saturated sets. However, those hyperbolic sets aren't necessarily basic pieces. This is because the domains of two basic pieces related by the Smale order cannot be disjoint, because of theorem \ref{domaintheorem}. Nevertheless, close to the point $p$ we'll see a cyclically ordered set of domains (see figure \ref{domainsaroundfixedpoint}), whose union contains all saddle-type basic pieces related to $p$.
\begin{figure}[h!]
\begin{center}
\includegraphics[scale=0.7]{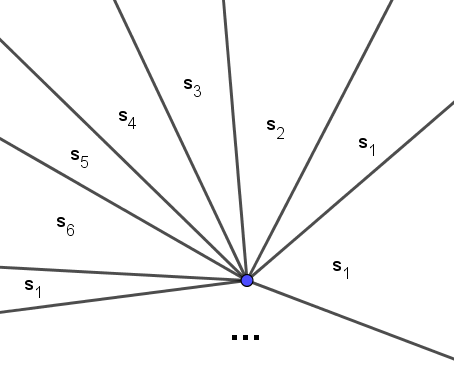}
\caption{}
\label{domainsaroundfixedpoint}
\end{center}
\end{figure}
\par{When applying $f$ to the previous image, the same cyclic order is obtained around $f(p)$. When applying $f^m$ to the previous image, where $m$ is the period of $p$, the image is either fixed or rotated. For instance, if the cyclic order is of the form $s_1,s_2,s_1,s_2$, applying $f^m$ can either send the first occurrence of $s_1$ to itself or to its second occurrence. If the image is rotated, the cyclic order is periodic as in the previous example. Again, by taking $f^{nm}$ for some $n \in \mathbb{N}$, the cyclic order around $p$ is fixed and the boundary lines of every domain are preserved. Let us now make some remarks.}
\begin{enumerate}
\item[•] A domain of a saturated hyperbolic set can indeed appear around an attractive or repelling periodic point many times.
\item[•] In figure \ref{domainsaroundfixedpoint}, the domain of $s_1$ appears two consecutive times in the cyclic order. This is also possible and it means that when applying theorem \ref{unglue} in order to create $f$, we need to glue together two boundary lines of the domain of $s_1$.
\item[•] Every saddle type basic piece related to $p$ must appear in exactly one of the domains surrounding $p$.
\item[•] In the above figure, the domains $s_1$ and $s_2$ are glued together and that means by theorem \ref{domaintheorem} that close to the line of intersection, $f^{nm}$ acts as a translation going from a unique hyperbolic repeller to our attracting point $p$. Therefore, inside the hyperbolic sets of $s_1$ and $s_2$, there exist two basic pieces (one in $s_1$ and one in $s_2$) that have a common ancestor in the Smale order. 
\item[•] For every domain intersecting $p$ there exists a basic piece inside the domain bigger than $p$ in the Smale order. However, not every basic piece inside the domain is bigger than $p$.
\end{enumerate}
\par{Let us now look more closely inside the domains $s_i$. We can assume without any loss of generality that the domains $s_i$ cannot be written as union of two disjoint domains glued along their boundaries. Therefore, by number 7 of the theorem \ref{domaintheorem}, if we represent the Smale order of $f$ as a graph, the subgraph of all the (saddle-type) basic pieces inside any $s_i$ is connected. Furthermore, by our previous remarks, for any two domains $s_i$ and $s_j$ that have been glued along their boundaries, there exist basic pieces $S_i \in s_i$, $S_j \in s_j$ and a repelling periodic point $a$, such that $a\rightarrow S_i,S_j \rightarrow p $. We'll show in the next section that the subgraph of all the basic sets bigger than $p$ is connected. This condition is a necessary condition that the Smale order of any structurally stable diffeomorphism must satisfy and will be called the \emph{connectivity condition}.}
\begin{definition}
Take a partial order on a finite set $(F,<)$ and represent it by a (non-oriented) graph $G$. We'll say that the partial order satisfies the connectivity condition, if for every maximal element $A$ the subgraph of $G$ of all the smaller elements than $A$ is connected and if for every minimal element $B$ the subgraph of $G$ of all the bigger elements than $B$ is connected.  
\end{definition}
\par{Assuming the necessity of the connectivity condition, we can construct a partial order that cannot be realised by a diffeomorphism with trivial hyperbolic attractors and repellers. For instance, consider the partial order given by the graph in figure \ref{impossibleorder}. The same conventions that we used in section 3 apply to this graph.}
\begin{figure}[h!]
\begin{center}
\includegraphics[scale=0.6]{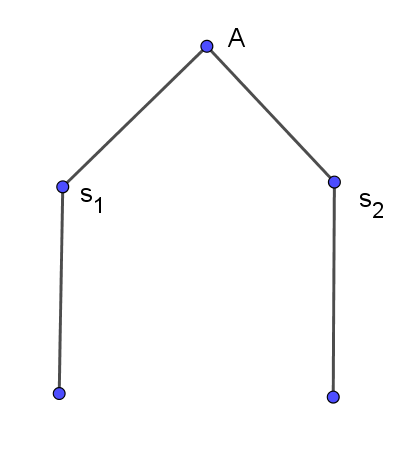}
\caption{}
\label{impossibleorder}
\end{center}
\end{figure}
\par{For the example in figure \ref{impossibleorder} it's quite easy to explain why the connectivity condition is necessary. If the unique hyperbolic repeller $A$ were a fixed point (being periodic would change nothing to our arguments), the domains of the saddle-type pieces $s_1$ and $s_2$ would intersect around $A$ somewhere, but the two pieces don't have a common descendant, therefore there couldn't be any translation band around the intersection of the two domains, which is absurd.}
\par{Let us remark that if $A$ were a non-trivial repeller, then using the methods described in section 3, it wouldn't be hard to show that the order in figure \ref{impossibleorder} is realisable by a stable diffeomorphism.}
\subsection*{The connectivity condition is necessary}
Take a stable orientation-preserving diffeomorphism $f$ acting on a closed oriented surface $S$ with trivial attractors and repellers. In this section, we'll prove that the Smale order of $f$ satisfies the connectivity condition. In order to do so, we'll prove the two following lemmas:
\begin{lemme}\label{lemma1}
Take $p$ an attractive periodic point of $f$. Fix $s$ a hyperbolic set that doesn't contain any attractors or repellers, whose extended domain intersects $p$ and cannot be written as union of two domains glued along their boundaries. For every basic piece $s_i \subset s$ bigger than $p$ in the Smale order of $f$, there exists a basic piece $s_j \subset s$ satisfying $s_i \rightarrow s_j \rightarrow p $ such that $s_j$ contains a periodic point with a free unstable separatrix (i.e. the separatrix doesn't intersect any stable manifold of $s$) intersecting only the basin of attraction of $p$.
\end{lemme}
Using the previous lemma, we know that there exist free separatrices crossing the stable manifold of every attracting periodic point. By 2 and 3 of theorem \ref{th-sbords}, the number of free separatrices attracted by any periodic point is finite. Also, it's not difficult to see that for every free separatrix $w$ intersecting the basin of attraction of an attracting point $p$, the set $w \cup \lbrace p \rbrace$ is a continuously embedded arc in $S$. Therefore, we can naturally cyclically order the separatrices arriving at any attractive periodic point.
\begin{lemme}\label{lemma2}
If one free separatrix issued from the basic piece $s$ comes in the above cyclic order right after one free separatrix issued from the basic piece $s'$ then there exists a repelling periodic point $\alpha$ bigger than $s$ and $s'$ in the Smale order of $f$.
\end{lemme}
Before moving to the proofs of the two previous lemmas, let us first prove that those lemmas imply that the Smale order of $f$ satisfies the connectivity condition. Indeed, take $p$ a periodic attracting point of $f$ and $s_i$, $s_j$ two distinct basic pieces bigger than $p$. By lemma \ref{lemma1}, there exist $s'_i$, $s'_j$ satisfying $s_i \rightarrow s'_i \rightarrow p$ and $s_j \rightarrow s'_j \rightarrow p$ such that both $s'_i$ and $s'_j$ have a periodic point with a free unstable separatrix crossing only the stable manifold of $p$. Thanks to the finiteness of the free separatrices, by a repeated application of lemma \ref{lemma2}, we obtain a path of the form $s'_i=S_0 , ~ \alpha_0, ~ S_1, ~ \alpha_1 , ~ S_2 ... , ~ \alpha_n , ~ S_n=s'_j$, with $\alpha_i$ being repelling periodic points and $S_i$ saddle-type basic pieces verifying $\alpha_i \rightarrow S_i, S_{i+1}\rightarrow p$. Therefore, $s'_i$ and $s'_j$ are connected by a path in the (non-oriented) graph of the Smale order of $f$, which implies the same for $s_i$ and $s_j$. We do the same for all other attracting and repelling periodic points of $f$ and we obtain that the Smale order of $f$ satisfies the connectivity condition.
\begin{proof}[Proof of lemma \ref{lemma1}]
\par{Take $s$ a hyperbolic set without any attractors or repellers, whose extended domain intersects $p$. Take $s_i\subset s$ a basic piece such that $s_i \rightarrow p$. Consider $e: \mathbb{R} \rightarrow S$ an immersion of one unstable manifold of $s_i$ crossing the stable manifold of $p$. The pre-image by $e$ of the basin of attraction of $p$ is an open $O$, therefore a union of open intervals. Take $(a,b)$ one such interval. By the phase theorem, $e(a)$ and $e(b)$ lie on the stable manifolds of two saddle-type points $p_1,p_2 \in \Omega(f)$ ($p_1$ and $p_2$ aren't necessarily distinct and belong to basic pieces smaller or equal to $s_i$ in the Smale order). Furthermore, $W^s(p_1),~ W^s(p_2)$ are s-boundaries (see theorem-definition \ref{def-sboundary}), therefore the points $p_1,p_2$ can be taken periodic with at least one free unstable separatrix by theorem \ref{th-sbords}. See figures \ref{arche1} and \ref{arche2} for two examples.}
\newline{}
\newline{}
\newline{}
\begin{figure}[h]
\centering
\begin{minipage}{.45\textwidth}
 \begin{flushleft}

  \includegraphics[scale=0.7]{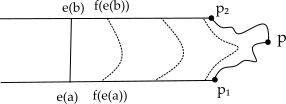}
  \captionof{figure}{}
  \label{arche1}
 \end{flushleft}
 \end{minipage}
\begin{minipage}{.4\textwidth} 
  \vspace{-1.7cm}
  
  \includegraphics[scale=0.7]{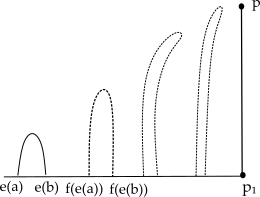}
  \captionof{figure}{}
   \label{arche2}
 \end{minipage}
\end{figure}
\par{Without any loss of generality, let us assume that $p_1,p_2$ are fixed points. Since $e(a) \in W^s(p_1)$ and $e(b) \in W^s(p_2)$, for $n>0$ the image of the arc $e((a,b))$ by $f^n$ approaches one free separatrix of $p_1$ and one of $p_2$, say $W^{u,+}(p_1)$ and $W^{u,+}(p_2)$ respectfully. Even more, since by definition $e(a,b)$ is inside the basin of attraction of $p$, according to Proposition 2.5.2 of \cite{Bonatti}, we have that when $n\rightarrow +\infty$, the set $f^n(e((a,b)))$ tends (in the Hausdorff sense) to $\lbrace p_1,p_2, p\rbrace \sqcup W^{u,+}(p_1) \sqcup W^{u,+}(p_2)$. From that we can easily deduce that $W^{u,+}(p_1)$ and $W^{u,+}(p_1)$ intersect the basin of attraction of $p$, which concludes our proof.}
\end{proof}
\begin{proof}[Proof of lemma \ref{lemma2}]
Fix an attractive periodic point $p$. We'll consider two different cases. Consider first two free separatrices coming the one right after the other in the cyclic order around $p$ issued from two disjoint hyperbolic saturated sets $s,s'$ without any attractors or repellers. Let us denote the two free separatrices $W^{u,+}(p_1)$ and $W^{u,+}(p_2)$. We remind that because of the theorem \ref{th-sbords}, $p_1 \in s$ and $p_2 \in s'$ can be taken periodic. From proposition 3.1.2 in \cite{Bonatti} and the definition of the domain (definition 3.2.1) in \cite{Bonatti} we have that there exists a path connecting $W^s(p_1)$ to the band of translation around the intersection of the two domains $s,s'$ without crossing the unstable manifold of any saddle-type basic piece. Same for $W^s(p_2)$. Therefore, if the band of translation connects the repelling periodic point $\alpha$ to $p$, then $W^u(\alpha)$ intersects $W^s(p_1)$ and $W^s(p_2)$. 
\par{Suppose now that we're given two free separatrices, one following the other in the cyclic order around $p$, that are issued from the same hyperbolic saturated set $s$ that doesn't contain any attractor or repeller. If $p_1$ and $p_2$ belong to the same basic piece, then the result of the lemma is trivial. If between the two free separatrices (inside the interval that contains no other free separatrices) lies a band of translation, we can apply in the same way our previous arguments. Otherwise, again thanks to proposition 3.1.2, the definition 3.2.1 and corollary 2.4.3 of \cite{Bonatti}, we have that we're in the situation of the figure \ref{arche1}. Following the notations of the figure \ref{arche1}, there exists an unstable manifold intersecting $W^s(p_1)$ and $W^s(p_2)$. Therefore, the basic set containing $p_1$ and the basic set containing $p_2$ have a common ancestor, which finishes the proof of the lemma.}
\end{proof}
\subsection*{The connectivity condition is sufficient}
From now on, let us call a Smale diffeomorphism, a stable diffeomorphism with trivial attractors and repellers. 

\begin{lemme}
For every $n_1,n_2,...,n_s \in 2\mathbb{N}-\lbrace 0, 2, 4\rbrace$ such that 
\begin{enumerate}
\item $\sum_{i=1}^s n_i \equiv 4s \text{ mod } 8$ 
\item $(s|n_1,...,n_s) \neq (2|6,10)$
\end{enumerate}
there exists a hyperbolic transitive piece whose extended domain (see corollary \ref{extendeddomain}) has exactly $s$ boundary components of respective lengths $n_1,n_2,...,n_s$ (see definition \ref{lengthofextendeddomain}).
\end{lemme}
\begin{proof}
Consider $p_i=n_i/2$. It's proved in \cite{Masur} that a pseudo-anosov map with $s$ singularities having respectfully $P_1,..., P_s \geq 3$ stable separatrices is realisable on a closed surface if and only if 
\begin{enumerate}
\item $\sum_{i=1}^s (P_i-2) \equiv 0 \text{ mod } 4$ 
\item $(s|P_1,...,P_s) \neq (2|3,5)$
\end{enumerate}
Therefore, there exists a closed surface $S$ and a pseudo-anosov map $f \in \text{Hom}^{+}(S)$ with $s$ singularities having respectfully $p_1,...,p_s$ stable separatrices. Similarly to DA maps defined for saddle points, we could define a DA map for $p$-pronged singular points. Using this adaptation of the DA map, we can open up the unstable separatrices of every $p$-pronged singular point as in figure \ref{unstablesingular} and for every $p$-pronged singularity we obtain $p$ new saddle points and a new repelling point. Apply the DA attracting map for every new saddle point obtained in the previous step as in figure \ref{stablesingular}.
\begin{figure}[h]

\begin{minipage}{.45\textwidth}
 \begin{flushleft}

  \includegraphics[width=1\linewidth]{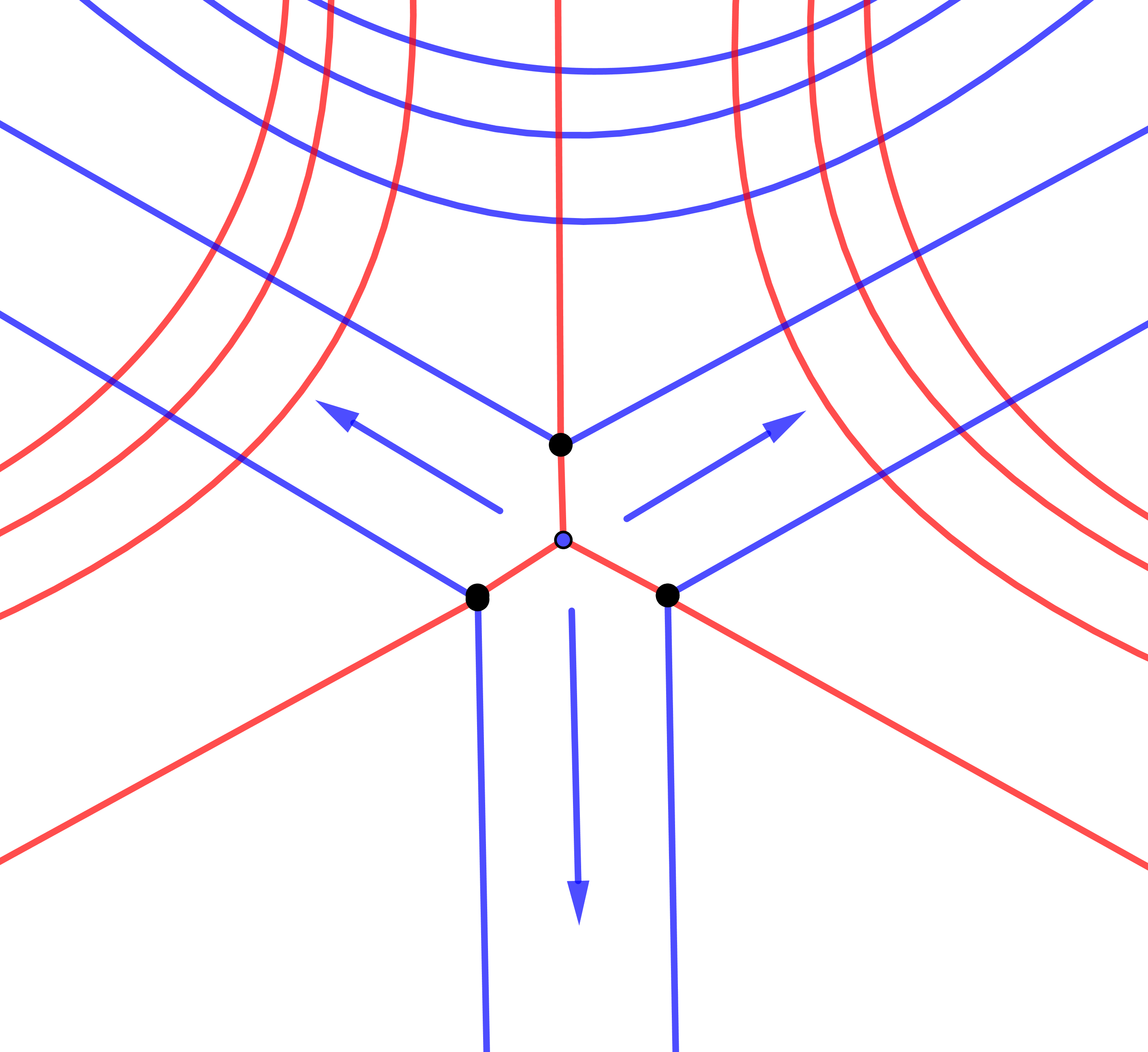}
  \captionof{figure}{}
  \label{unstablesingular}
 \end{flushleft}
 \end{minipage}
\hspace{1cm}
\begin{minipage}{.4\textwidth} 
\begin{flushright}
\includegraphics[scale=0.075]{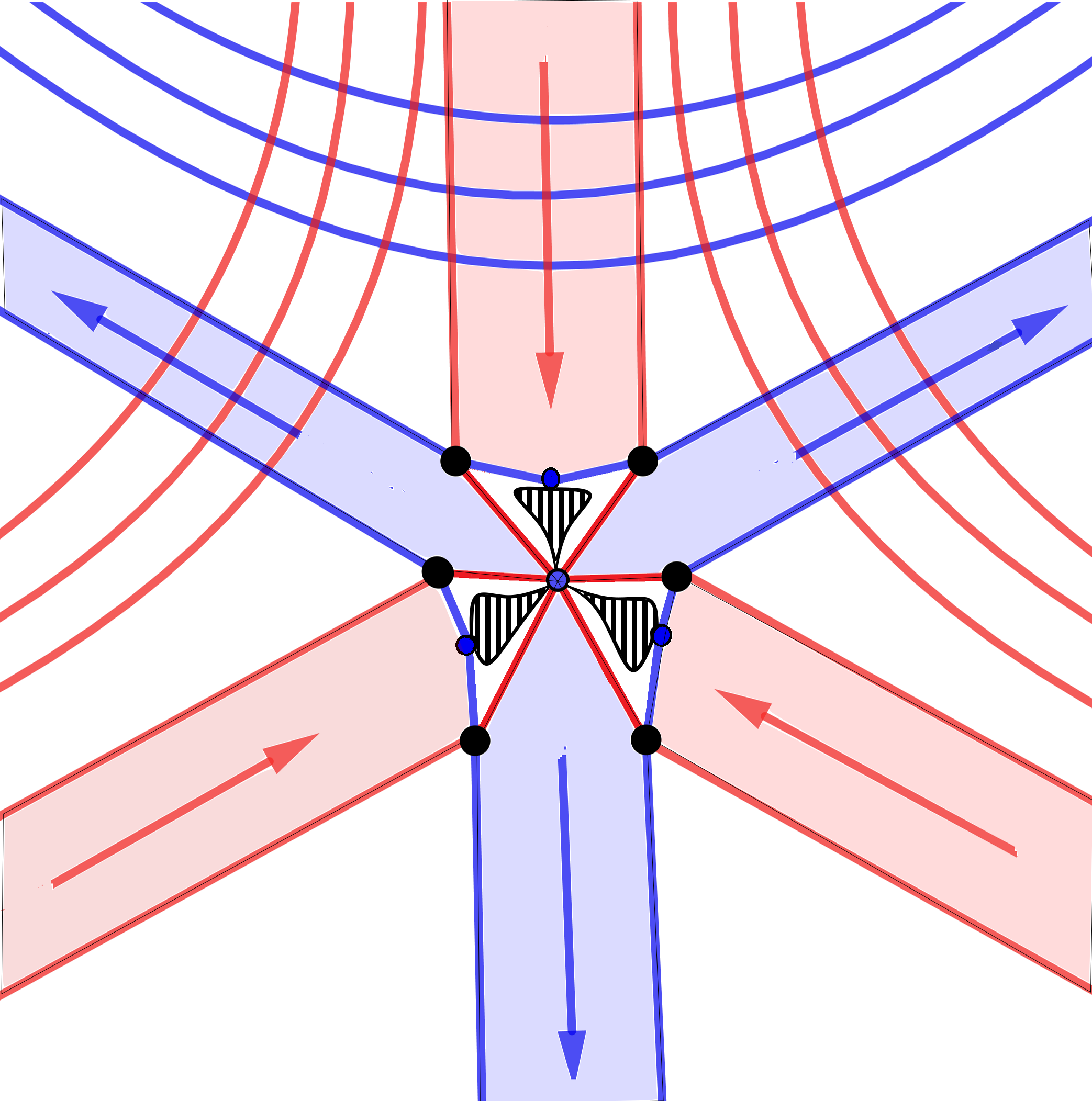}
  \vspace{-0.6cm}
  \hspace{0.5 cm}
  \caption{}
   \label{stablesingular} 
\end{flushright}
\end{minipage}  
\end{figure}
\par{By the previous method, we obtain a new homeomorphism $f'$ and by smoothing $f'$, we obtain a new hyperbolic transitive set $K$, whose extended domain has $s$ boundary components of respective lengths $n_1,n_2,...,n_s$. Indeed, the dynamics outside the bands between the opened separatrices, is semiconjugated to the initial dynamics after the DA perturbation and even after smoothing $f'$ the hyperbolicity and transitivity are conserved. Inside the new bands, every point is wandering except from the repelling and attracting fixed points created by the DA perturbations. Also, by our discussion in section 2.3, every point inside a blue band belongs to some stable manifold of $K$ and every point inside a red band belongs to some unstable manifold of $K$. However, the open disks bounded by quadrilaterals, whose boundaries consist of two blue and two red lines in figure \ref{stablesingular} contain no invariant manifold of $K$. Therefore, using that the domain is open and that it contains all the invariant manifolds of $K$, the extended domain of $K$ will be of the form drawn in figure \ref{stablesingular}, where the disks with the line pattern represent its complement. Finally, the boundary of every black disk consists of two lines, close to which $f'$ acts as a translation from a repelling fixed point to an attracting one. Hence, to every opened singularity corresponds a boundary component of the extended domain of $K$ with length equal to two times the valency of the singularity, which proves that the boundary components of the extended domain of $K$ have the prescribed lengths.}
\end{proof}
Without any loss of generality, assume that the previously constructed hyperbolic set contains many fixed saddle-type points. For every one of them we can perform the same DA process described in the previous proof and create as many boundary components of length 4 as we want. We can therefore generalise the previous lemma to the following: 
\begin{lemme}\label{creation}
For every $n_1,n_2,...,n_s \in 2\mathbb{N}-\lbrace 0\rbrace$ such that 
\begin{enumerate}
\item $\sum_{i=1}^s n_i \equiv 4s \text{ mod } 8$ 
\item $(n_1,...,n_s) \neq (10,6,4,4,4,...)$
\end{enumerate}
there exists a hyperbolic transitive piece whose extended domain has exactly $s$ boundary components of respective lengths $n_1,n_2,...,n_s$.
\end{lemme}
Let us make some remarks about the previous lemma:
\begin{enumerate}
\item In \cite{Masur} a stronger result than the one we used in our proof is shown: except a finite number of exceptions, given $P_1,..., P_s \geq 1$ with $P_i \neq 2$ there exists a pseudo-anosov map on a surface with a finite number of punctures having $s$ singularities with respective valencies $P_1,..., P_s$. Therefore, using the previous fact, the lemma \ref{creation} can be further generalised by applying the DA process described in the proof above to 1-pronged singularities.
\item Of course, applying a DA process to pseudo-anosov maps constitutes one of the possible constructions of domains of transitive hyperbolic pieces, but it's not the only one. This raises the following question: Given any $n_1,...,n_s \in 2\mathbb{N} -\lbrace 0\rbrace$ does there exist a basic piece whose extended domain has $s$ boundary components of respective lengths $n_1,...,n_s$?
\end{enumerate}
 
\begin{proof}[Proof of theorem \ref{Stablecase}]
\par{We will break down the following proof in two steps: 
\begin{enumerate}
\item Using the conventions of section 3, every partial order on a finite set, whose associated graph consists of three level of elements: maximal, minimal and first generation saddles is realisable as a Smale order of a Smale diffeomorphism
\item The previous step implies the general case 
\end{enumerate}}
Take a partial order $(F,<)$ satisfying the conditions of Step 1. In the following lines, we'll always denote a maximal element by $\alpha$, a minimal one by $\omega$ and any other element by $s$. 
\subsubsection*{Constructing locally the dynamics around repelling and attracting fixed points}
\par{The dynamics of a Smale diffeomorphism, whose Smale order satisfies the hypothesis of Step 1, can be obtained after gluing together domains of transitive hyperbolic sets. This is because, since the saddle-type basic pieces are not related by the Smale order, their domains can be taken disjoint by theorem \ref{domaintheorem}. Therefore, in this case, around every repelling or attracting periodic point of $f$, there exists a cycle of transitive domains (see figure \ref{domainsaroundfixedpoint}) and this motivates our definition of \emph{cycle}.}
\begin{definition}
Take a partial order on a finite set $(F,<)$, a minimal element $\omega$ and the set $\lbrace s^{\omega}_1, s^{\omega}_2,...,s^{\omega}_n\rbrace$ of all non-maximal elements bigger than $\omega$. We'll call $s^{\omega}_{i_1} \xrightarrow{\alpha_1} s^{\omega}_{i_2} \xrightarrow{\alpha_2} s^{\omega}_{i_3} \xrightarrow{\alpha_3} s^{\omega}_{i_4} \xrightarrow{\alpha_4}... \xrightarrow{\alpha_n}s^{\omega}_{i_1}$ a cycle of $\omega$ if all the $s^{\omega}_i$ and the maximal elements bigger than $\omega$ appear at least once in the sequence and for every $j$ we have $\alpha_j>s^{\omega}_{i_j},s^{\omega}_{i_{j+1}}$. Similarly, we can define cycles for repelling points. 
\end{definition} 
It's not difficult to see that a partial order satisfies the connectivity condition if and only if there exists at least one cycle around every maximal and minimal element. 
\newline{}
\begin{figure}[h!]
\begin{center}
\includegraphics[scale=0.5]{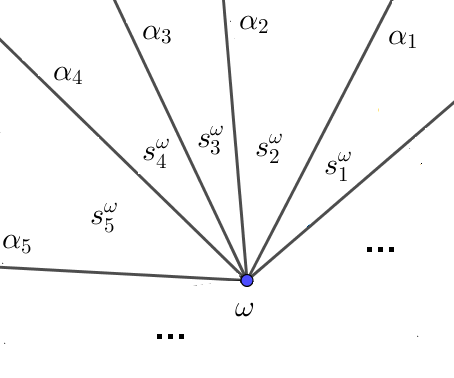}
\caption{}
\label{domainsaroundfixedpoint1}
\end{center}
\end{figure}
\par{Let us now represent as in figure \ref{domainsaroundfixedpoint1} every minimal element $\omega$ of $(F,<)$ by an attracting fixed point accompanied by a cycle of the form $s^{\omega}_1 \xrightarrow{\alpha_1} s^{\omega}_2 \xrightarrow{\alpha_2} s^{\omega}_3 \xrightarrow{\alpha_3} s^{\omega}_4 \xrightarrow{\alpha_4}... \xrightarrow{a_n}s^{\omega}_1$. We'll say that in figure \ref{domainsaroundfixedpoint1} the domains of the different basic pieces are delimited by \emph{bands}. The reader can think of $\alpha_i$ as the repelling point for which there exists a band of translation from $a_i$ to $\omega$ close to the intersection of $s_i$ and $s_{i+1}$, and of every band, as one half of one translation band between a repelling point and $\omega$. We assume, without any loss of generality, that the cycles around every minimal element $\omega$ satisfy the following conditions:}
\begin{enumerate}
\item For every $s_k,s_l,\alpha_k$ in $F$ such that $\alpha_k>s_k,s_l>\omega$ the transition $s^{\omega}_k \xrightarrow{\alpha_k} s^{\omega}_l$ appears at least once in the cycle.
\item For $s^{\omega}_k \neq s^{\omega}_l$ the number of transitions of the form $s^{\omega}_k \xrightarrow{\alpha_k} s^{\omega}_l$ is the same as the number of transitions of the form $s^{\omega}_l \xrightarrow{\alpha_k} s^{\omega}_k$
\end{enumerate}

\par{The first condition is always possible thanks to the connectivity condition and in order to assure the second condition, we can always add at the end of a cycle its symmetric. In the same way, we fix a cycle verifying the previous conditions for every maximal element. By convention, we'll draw the cycles in a counter clockwise way for attracting points and in a clockwise way for repelling ones. Such an orientation associates every band around any minimal element $\omega$ to its \emph{type}, an arrow of the form $s^{\omega}_k \xrightarrow{\alpha} s^{\omega}_l$. Same for bands around maximal points.}
\par{Choosing a cycle around repelling and attracting points fixes the local dynamics around of these points. In order that the previous choice of cycles be realised by a Smale diffeomorphism we need that for every minimal element $\omega$ the number of bands of type $s^{\omega}_l \xrightarrow{\alpha_j} s^{\omega}_k$ is equal to number of bands of type $s^{\alpha_j}_l \xrightarrow{\omega} s^{\alpha_j}_k$. Indeed, as we've said before, each band of the previous type corresponds to half a translation band going from $\alpha_j$ to $\omega$. We can always choose a collection of cycles around repelling and attracting fixed points that verifies the previous condition, that we'll call condition $(*)$. } 
\begin{lemme}\label{exists equilibrium}
There exists a collection of cycles around the minimal and maximal elements satisfying the conditions $1,2$ and $(*)$.
\end{lemme}
\begin{proof}
Take any collection of cycles around the minimal and maximal elements satisfying the conditions 1 and 2. This is always possible by our previous discussion. If for some $\omega, \alpha_j, s_k, s_l$ with $k \neq l$ we have that in our choice of cycles $\#(s^{\omega}_l \xrightarrow{\alpha_j} s^{\omega}_k) > \#(s^{\alpha_j}_l \xrightarrow{\omega} s^{\alpha_j}_k)$ (by condition 2 we also have that $\#(s^{\omega}_k \xrightarrow{\alpha_j} s^{\omega}_l) > \#(s^{\alpha_j}_k \xrightarrow{\omega} s^{\alpha_j}_l)$), then we apply the operation shown in figure \ref{equilibriu} to $\alpha_j$'s cycle.   
\begin{figure}[h!]
\begin{center}
\includegraphics[scale=0.35]{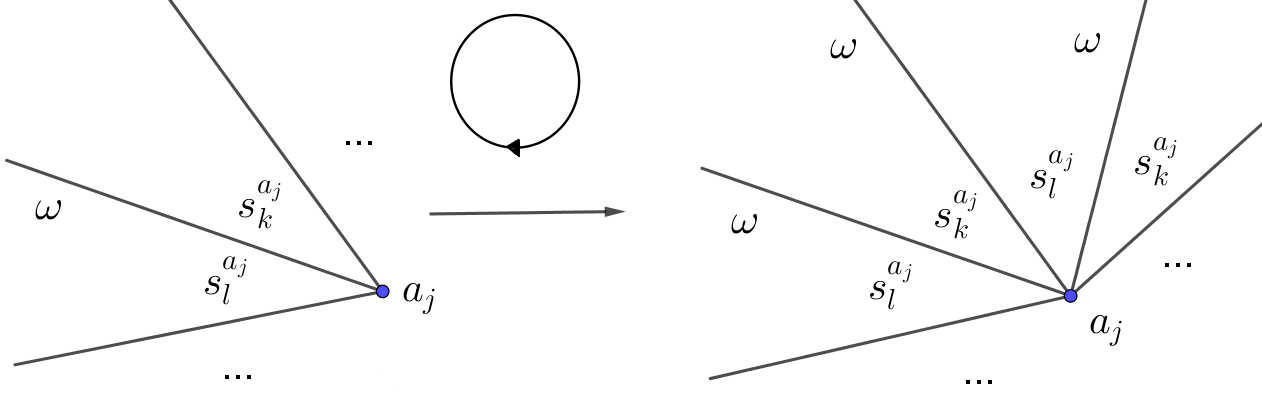}
\caption{}
\label{equilibriu}
\end{center}
\end{figure}
\par{The previous operation is possible thanks to condition 1 and the new cycle around $\alpha_j$ satisfies the conditions 1 and 2. After the previous operation, the number of $s^{\alpha_j}_l \xrightarrow{\omega} s^{\alpha_j}_k$ and $s^{\alpha_j}_k \xrightarrow{\omega} s^{\alpha_j}_l$ increases by one and by condition 2 we have that $\#(s^{\omega}_l \xrightarrow{\alpha_j} s^{\omega}_k)= \#(s^{\omega}_k \xrightarrow{\alpha_j} s^{\omega}_l)$ and $\#(s^{\alpha_j}_k \xrightarrow{\omega} s^{\alpha_j}_l)= \#(s^{\alpha_j}_l \xrightarrow{\omega} s^{\alpha_j}_k)$. Therefore, after a finite number of repetitions the two previous cardinals will be equal.}
\par{Similarly, if for some $\omega, \alpha_j, s$ we have that in our choice of cycles $\#(s^{\omega} \xrightarrow{\alpha_j} s^{\omega}) > \#(s^{\alpha_j} \xrightarrow{\omega} s^{\alpha_j})$, then we apply the operation shown in figure \ref{equilibriu1} to $\alpha_j$'s cycle.
\begin{figure}[h!]
\begin{center}
\includegraphics[scale=0.35]{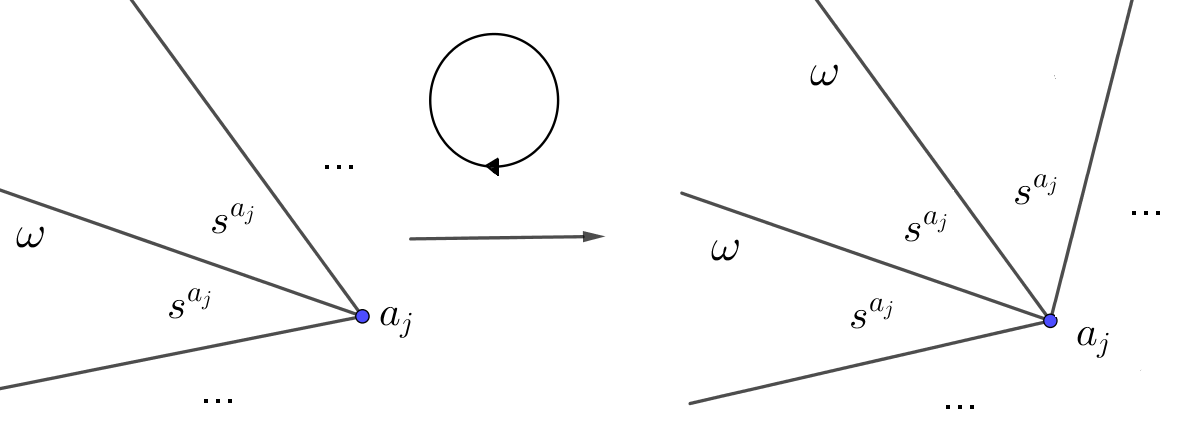}
\caption{}
\label{equilibriu1}
\end{center}
\end{figure}
Again, the previous operation is possible thanks to condition 1, the new cycle around $\alpha_j$ satisfies the conditions 1 and 2 and after a finite number of repetitions we get $\#(s^{\omega} \xrightarrow{\alpha_j} s^{\omega}) = \#(s^{\alpha_j} \xrightarrow{\omega} s^{\alpha_j})$.}
\end{proof}
Thanks to the previous lemma, we can assume in what follows that our choices of cycles around minimal and maximal elements verify for any $j,k,l,p$ 
\begin{equation}\label{equilibrium} \tag{*} 
\#s^{\omega_p}_k \xrightarrow{\alpha_j} s^{\omega_p}_l ~=~\#s^{\omega_p}_l \xrightarrow{\alpha_j} s^{\omega_p}_k ~=~\#s^{\alpha_j}_l \xrightarrow{\omega_p} s^{\alpha_j}_k~=~\#s^{\alpha_j}_k \xrightarrow{\omega_p} s^{\alpha_j}_l
\end{equation}
\par{The only thing now remaining is to glue the domains of the hyperbolic pieces, following our previous choice of cycles around attracting and repelling points. This process looks similar to completing a puzzle: by choosing cycles around fixed points we've put together all the boundary puzzle pieces and the only thing remaining is to fit the rest of the pieces inside the constructed frame.}
\subsubsection*{A first example}
\par{Before providing the general construction in the next section and in order to make the idea of the proof more clear to the reader, let us first look at a simple example. Suppose we have only one attracting and one repelling point and we've chosen cycles around them satisfying (\ref{equilibrium}) exactly as in figure \ref{example1}. We remind that we've chosen opposite orientations when drawing cycles around repelling and attracting points. This is not an arbitrary choice of ours, but a necessary precaution in order to preserve the orientation in our drawings. }
\par{In order to glue the domain of $s_1$ to our drawing, we have to first draw its boundary, therefore the translation bands from $\alpha$ to $\omega$. In order to do so, we glue together the bands around $\alpha$ and $\omega$. Once the boundary of the domain of $s_1$ has been determined, all that remains is to construct a domain for $s_1$ with a prescribed number of boundary components with prescribed lengths, which will be done in a second phase. Let us make the above more precise by using the example in figure \ref{example1}.}
\par{Start from the attractor $\omega$ and consider its band of type $s_2^{\omega} \xrightarrow{\alpha} s_1^{\omega}$. This band is part of a translation band beginning from $\alpha$ and arriving at $\omega$, therefore it must be glued to a band of the type $s_2^\alpha \xrightarrow{\omega} s_1^\alpha$ and since by our choice of cycles there is just one of them, we have no choice but to glue those two bands together. The result of this gluing is part of the boundary of the domain of $s_1$. Remember that the boundaries of the domains are star-shaped (see figure \ref{domain}), therefore in order to complete the boundary of the domain of $s_1$, we now need to decide to which band is glued the band $s_1^\alpha \xrightarrow{\omega} s_2^\alpha$ (the next in the cyclic orientation of $\alpha$ after $s_2^\alpha \xrightarrow{\omega} s_1^\alpha$). Once more, we don't have much of a choice we must glue it together with the unique band of $\omega$ of type $s_1^{\omega} \xrightarrow{\alpha} s_2^{\omega}$. Again, since the domains are star-shaped we now have to decide where to glue the band $s_2^{\omega} \xrightarrow{\alpha} s_1^{\omega}$ (the one before $s_1^{\omega} \xrightarrow{\alpha} s_2^{\omega}$ in the cyclic orientation of $\omega$), but this choice has already been made.}
\begin{figure}[h!]
\begin{center}
\includegraphics[scale=0.25]{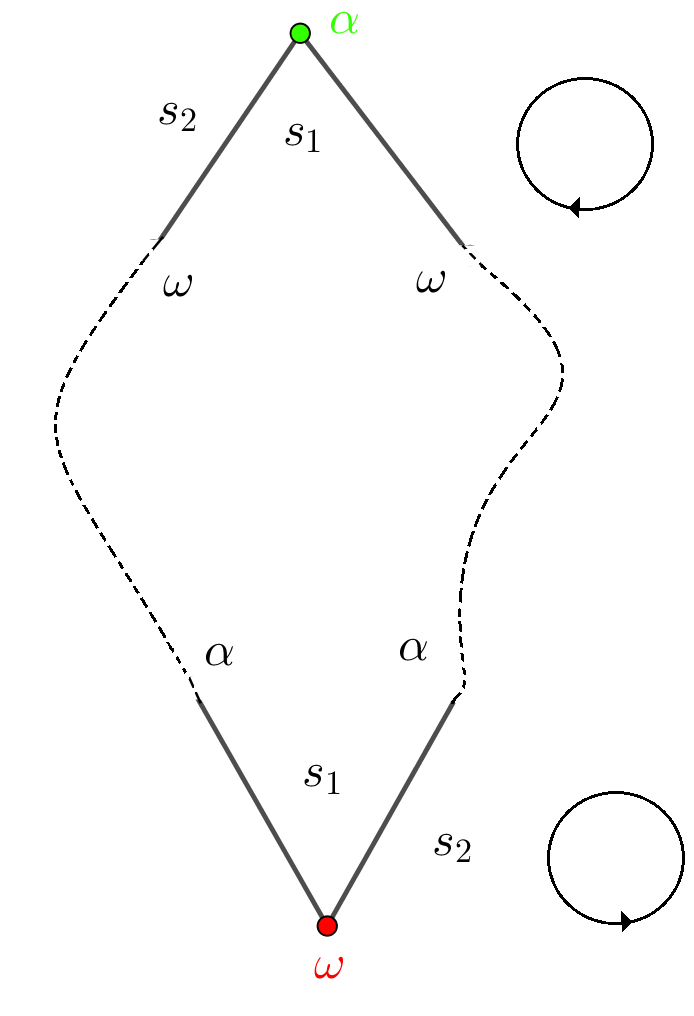}
\caption{}
\label{example1}
\end{center}
\end{figure}
\par{By using the above algorithm, which we're going to generalise in the next section, we've now formed a cyclic boundary component of $s_1$ and since we've exhausted all the bands related to $s_1$, the behaviour of the boundary of the domain of $s_1$ has been determined. We repeat the same process for $s_2$, but our previous choices determine completely the behaviour of the boundary of $s_2$. Therefore, now we move on to the gluing of the domains. There exists a transitive hyperbolic piece whose domain has a unique boundary component of length 2, namely the standard horseshoe on the sphere. Take the domains of two horseshoes and glue them together along their boundaries, according to the figure \ref{example1}. Thanks to theorem \ref{gluening}, at the end of this process we receive a Smale diffeomorphism. Furthermore, since our band gluings can be drawn in the plane and the domain of a horseshoe is planar, the surface that we obtain at the end of this process is a sphere.}
\subsubsection*{Generalisation}
\par{Let us now generalise the previous algorithm. For every attracting or repelling point $A$ consider an indexation in $\mathbb{Z}/ n\mathbb{Z}$ of the translation bands compatible with the cyclic orientations, where $n$ is the length of the cycle of bands around $A$. Denote by $S^{\omega}_i \xrightarrow{\alpha} S^{\omega}_j$ the set of all bands of type $s^{\omega}_i \xrightarrow{\alpha} s^{\omega}_j$ around $\omega$ and $S_i \rightarrow S_j:= \bigcup_{\alpha,\omega} (S^{\omega}_i \xrightarrow{\alpha} S^{\omega}_j \cup S^{\alpha}_i \xrightarrow{\omega} S^{\alpha}_j)$. A band in $S_i \rightarrow S_j$ will called of type \emph{beginning} for $s_j$ and of type \emph{end} for $s_i$ (see figure \ref{example1}). A \emph{boundary cycle} for $s_2$ is a finite sequence of bands of the form 
\begin{align*}
(s^{\omega_1}_1 \xrightarrow{\alpha_1} s^{\omega_1}_2)_{i_1} &\longleftrightarrow(s^{\alpha_1}_1\xrightarrow{\omega_1} s^{\alpha_1}_2)_{i_2} \hookrightarrow (s^{\alpha_1}_2 \xrightarrow{\omega_2} s^{\alpha_1}_3)_{i_3} \longleftrightarrow (s^{\omega_2}_2\xrightarrow{\alpha_1} s^{\omega_2}_3)_{i_4}  \\& \hookrightarrow (s^{\omega_2}_4 \xrightarrow{\alpha_2} s^{\omega_2}_2)_{i_5} \longleftrightarrow ... \hookrightarrow (s^{\omega_1}_1 \xrightarrow{\alpha_1} s^{\omega_1}_2)_{i_1}
\end{align*} that satisfies the following axioms:
\begin{enumerate}
\item[•] For every $i,j,k,l,m$ we have that $(s^{\omega_i}_j \xrightarrow{\alpha_k} s^{\omega_i}_l)_{i_m} \in S^{\omega_i}_j \xrightarrow{\alpha_k} S^{\omega_i}_l $
\item[•] If in the above sequence a band $e$ of type end for $s_2$ follows a band $b$ of type beginning for $s_2$ ($b\hookrightarrow e$), then the two bands belong to the same repelling point. If a band of type beginning $b$ follows a band of type end $e$ ($e\hookrightarrow b$), then the two bands belong to the same attracting point. Furthermore, in both of the above cases, if the indices of $e$ and $b$ are $i(e)$ and $i(b)$ we have $i(e)=i(b) + 1$. 
\item[•] If we identify the first and last band of the sequence, every band in $S_2 \rightarrow S_k$ or $S_k \rightarrow S_2$ with $k \neq 2$ must appear at most once inside the sequence. Similarly, every band in $S_2 \rightarrow S_2$ must appear at most twice.
\end{enumerate}}
The definition of a boundary cycle for $s_1$ can seem rather complicated, but it is not more than a formalisation of the algorithm that we've used in the example described in the previous section to form the domain of $s_1$: start from a band $s$ of type beginning for $s_1$ around an attracting point $\omega$, it corresponds to a translation band from a repelling point $\alpha$ to $\omega$, glue $s$ to some band $s'$ of the same type around $\alpha$ ($s \longleftrightarrow s'$), consider the band $s''$ coming after $s'$ for the cyclic orientation around $\alpha$ ($s' \hookrightarrow s''$), $s''$ is of type end for $s_1$ and it corresponds to a translation band from $\alpha$ to some attracting point $\omega_1$, glue it to a band $s'''$ of $\omega_1$ of the same type ($s''\longleftrightarrow s'''$), consider the band coming before $s'''$ in $\omega_1$, etc. The previous algorithm doesn't have the right to glue a new band to an old band that has already been glued and continues in this way until the band $s$ from which we started is reached. The condition (\ref{equilibrium}) is essential in order that this algorithm be well-defined.
\par{A boundary cycle for $s_1$ represents a boundary component of the domain of $s_1$. If we identify the first and last elements of a given boundary cycle, the number of its elements divided by two will be called the \emph{length of the boundary cycle}. The length of a boundary cycle is always even and represents the length of its associated boundary component.}
\par{Let us now look into a more sophisticated example. Once again around the attracting and repelling points we've chosen cycles that satisfy (\ref{equilibrium}). 
\begin{figure}[h!]
\begin{center}
\includegraphics[scale=0.35]{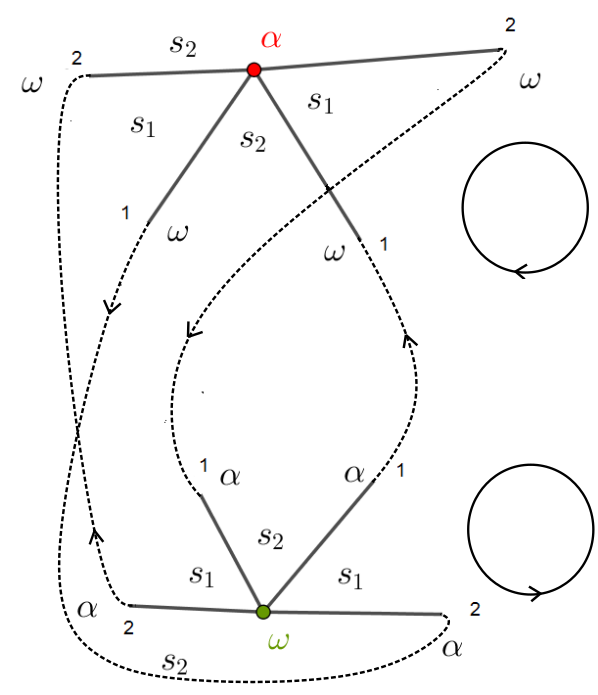}
\caption{}
\label{example}
\end{center}
\end{figure}
We glue the bands together as drawn in figure \ref{example}. Every gluing of the bands determines a partition of the bands in boundary cycles. In our case, the partition corresponding to $s_2$ consists of one cycle, namely: }
\begin{align*}
(s^{\omega}_1 \xrightarrow{\alpha} s^{\omega}_2)_{1} &\longleftrightarrow (s^{\alpha}_1\xrightarrow{\omega} s^{\alpha}_2)_{1} \hookrightarrow (s^{\alpha}_2 \xrightarrow{\omega} s^{\alpha}_1)_{1} \longleftrightarrow (s^{\omega}_2\xrightarrow{\alpha} s^{\omega_2}_1)_2 \hookrightarrow (s^{\omega}_1 \xrightarrow{\alpha} s^{\omega}_2)_2 \\ &\longleftrightarrow (s^{\alpha}_1 \xrightarrow{\omega} s^{\alpha}_2)_2 \hookrightarrow (s^{\alpha}_2 \xrightarrow{\omega} s^{\alpha}_1)_2 \longleftrightarrow (s^{\omega}_2 \xrightarrow{\alpha} s^{\omega}_1)_1 \hookrightarrow (s^{\omega}_1 \xrightarrow{\alpha} s^{\omega}_2)_1
\end{align*}
The arrows drawn in figure \ref{example} represent the order of appearance of every band inside the boundary cycle of $s_2$. Once again our choices for $s_2$ determine completely the gluings of all the bands related to $s_1$. Each of the domains $s_1$ and $s_2$ has according to our band gluings, one cyclic boundary component of length 4. We could take $s_1$ and $s_2$ to be fixed saddle points and glue them along their boundaries following figure \ref{example}. Once again, thanks to theorem \ref{gluening}, at the end of this process we receive a Smale diffeomorphism. The difference of this example with the previous one is that the band gluings in this example cannot be drawn in a planar way. Therefore, despite the fact that the domains of $s_1$ and $s_2$ have been chosen planar, the Smale diffeomorphism that we construct at the end acts on the torus and the dynamics can be identified to the time 1 of the stable gradient flow on the torus. 

\par{Let us move to the general case. We remind that in the previous section we've chosen cycles of bands verifying what we've called conditions 1, 2 and (\ref{equilibrium}) (see lemma \ref{exists equilibrium}). Start from the bands related to the basic piece $s_1$ and glue for every $\omega, \alpha, k$ every band in $S_1^{\omega}\xrightarrow{\alpha}S_k^{\omega}$ (resp. $S_k^{\omega}\xrightarrow{\alpha}S_1^{\omega}$) to a band in $S_1^{\alpha}\xrightarrow{\omega}S_k^{\alpha}$ (resp. $S_k^{\alpha}\xrightarrow{\omega}S_1^{\alpha}$). We'll call this a \emph{band gluing}. It's not difficult to see that if the condition (\ref{equilibrium}) is satisfied, then a band gluing is always possible and groups together all the bands related to $s_1$ inside a finite number of boundary cycles for $s_1$. Every boundary cycle for $s_1$ will represent a boundary component of the domain of $s_1$. If we identify the first and last element of the previous boundary cycles:}
\begin{itemize}
\item A band in $S_1 \rightarrow S_k$ with $k \neq 1$ appears once inside exactly one boundary cycle.
\item A band in $S_1 \rightarrow S_1$ can appear twice inside the same boundary cycle or once in exactly two different boundary cycles (compare figures \ref{cycle2cycle10} and \ref{example2})
\end{itemize}
\par{Next, keeping in mind that we've already glued the bands related to $s_1$, in the exact same way we glue together the bands related to $s_2$ by respecting the previous gluings, therefore we obtain a grouping of all the bands related to $s_2$ in a finite number of boundary cycles. Once again any band in $S_2 \rightarrow S_2$ could appear twice in the same boundary cycle or once in two different boundary cycles. We continue gluing bands in this way for all $s_i$. At the end of this process, we've fixed the behaviour of the boundary of every domain $s_i$ and more specifically the way in which we're going to glue the domains of the $s_i$ together. }
\begin{figure}[h!]
\begin{center}
  \includegraphics[scale=0.4]{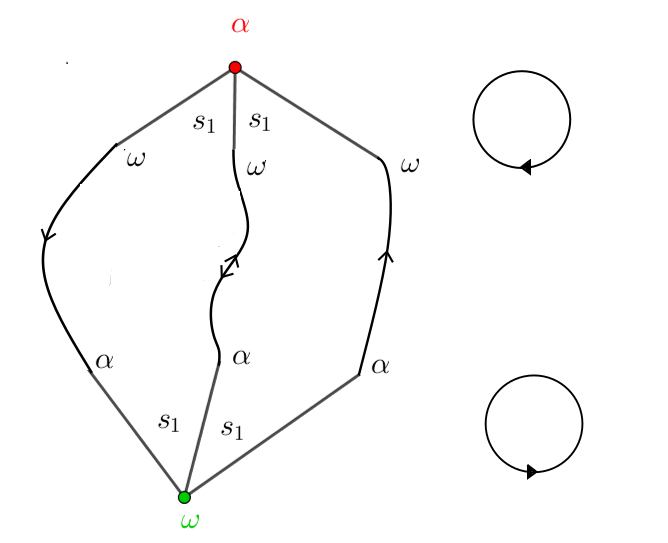}
  \captionof{figure}{}
   \label{example2}
\end{center}
\end{figure}
\subsubsection*{Avoiding obstructions}
\par{The only thing to do now is to construct for every $s_i$ a domain that has the same number of boundary components as the number of boundary cycles corresponding to $s_i$ and whose boundary components have lengths given by the lengths of the respective boundary cycles. The only problem is, that the lemma \ref{creation} constructs domains, whose boundary lengths satisfy certain conditions. Take for instance the boundary cycles defined for $s_1$. If one of them is of length 2, by changing a little bit our initial choice of cycles around the attracting and repelling points, we can apply the method shown in figure \ref{cycle2cycle10} to turn it into a boundary cycle of length 10:  

\begin{figure}[h!]
\begin{center}
  \includegraphics[scale= 0.35]{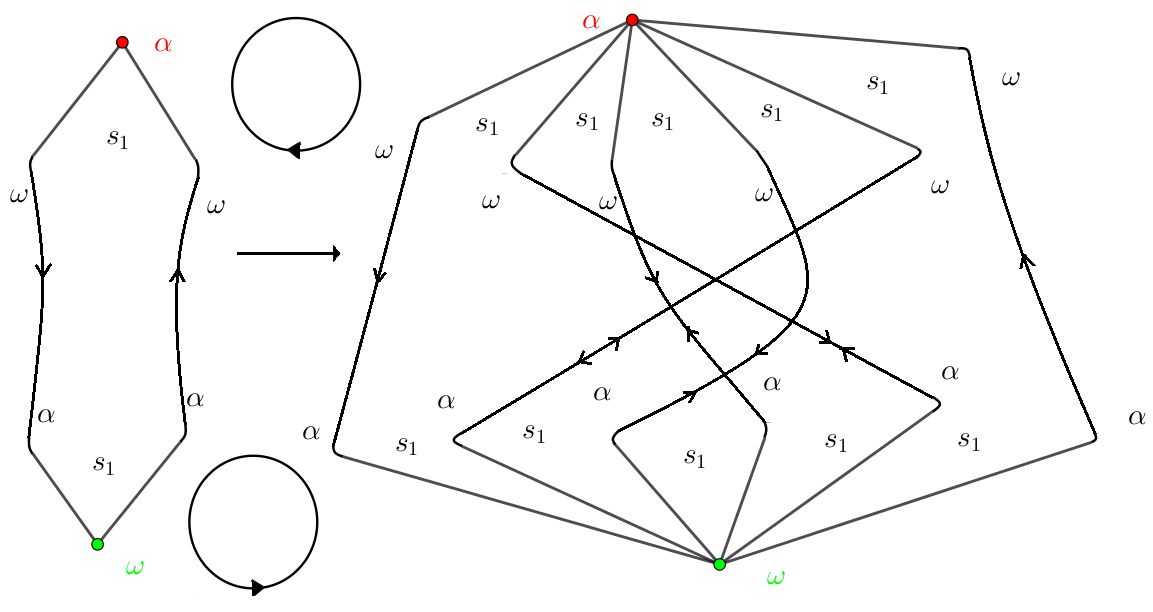}
  \captionof{figure}{}
  \label{cycle2cycle10}
\end{center}
\end{figure}
Next denote the number of boundary cycles for $s_1$ by $s$ and their respective lengths as $n_1,...,n_s$. By our previous trick, we can assume that $n_1,...,n_s \geq 4$. If $(\sum_{i=1}^s n_i) -4s \not\equiv 0 \text{ mod } 8$, then since every boundary cycle has even length we have $(\sum_{i=1}^s n_i) -4s \equiv 2,4 \text{ or }6 \text{ mod } 8$. Again by changing a little bit our initial choice of cycles around the attracting and repelling points as shown in figure \ref{breakcycles}, we can increase the total number of cycles by 1 and their total lengths by 2.

In this way, if $s'$ is the new number of cycles and $n_1',...,n_{s'}'$ their respective lengths we have  $(\sum_{i=1}^{s'} n'_i) -4s' \equiv 0, 2 \text{ or }4 \text{ mod } 8$. If any of those cycles has now length 2, then repeat the trick described before that increases its length by 8. Therefore, after a finite number of repetitions of this trick the condition 1 of lemma \ref{creation} can be satisfied for every domain of every $s_i$. If the condition 2 of the lemma isn't satisfied for some $s_i$, then we apply the same trick.}
\par{ Having satisfied all the hypothesis of the lemma \ref{creation} we can now construct a domain for every $s_i$ with the prescribed number of boundary components and the prescribed lengths. By gluing, all those domains along the bands that we've glued, we obtain a Smale diffeomorphism (by theorem \ref{gluening}) realising the order $(F,<)$ as its Smale order. Indeed, for every attracting point $A$ (corresponding to a minimal point of $F$), every $s$ satisfying $s>A$ has its domain glued to $A$, whereas the domain of every $s' \in F$ not related to $A$ doesn't intersect $A$, by our initial choice of cycles around $A$. Same for every repelling point $B$ and since we've assumed for this part of the proof that there is no relation between saddles, our previous construction wields a realisation of $(F,<)$.}
\begin{figure}[h!]
\begin{center}
\includegraphics[scale=0.2]{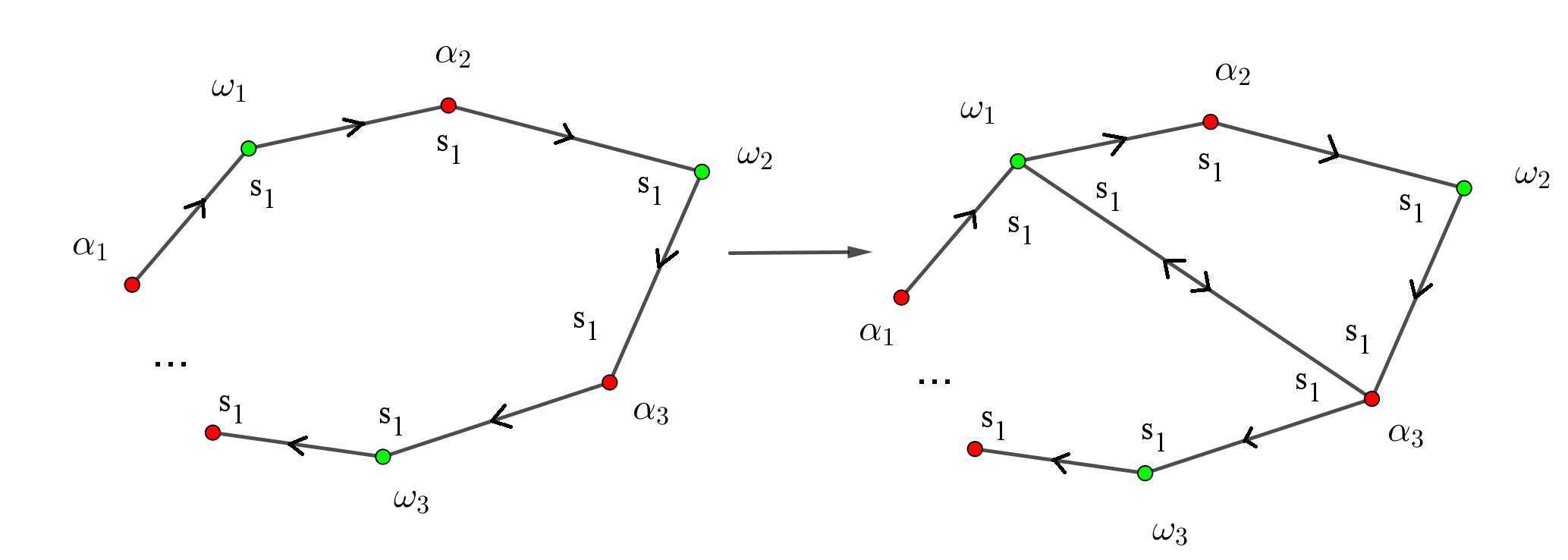}
\caption{}
\label{breakcycles}
\end{center}
\end{figure} 
\subsubsection*{Step 1 implies the general case}
\par{Let us now finish the proof by considering a general order $(F,<)$ satisfying the connectivity condition. Consider now a new order $<_1$ on $F$, by deleting all relations between any two non-maximal and non-minimal elements of $<$. Therefore, any non-maximal and non-minimal element of $F$ is related by $<_1$ to the same minimal or maximal elements as before, but not to any other non-maximal and non-minimal element. It is easy to verify that the new order $<_1$ also satisfies the connectivity condition. Therefore, thanks to the proof of Step 1, we can realise $(F,<_1)$ as a Smale order of some Smale diffeomorphism $f$ acting on a closed surface $S$.} 
\par{Take now $s_1$ and $s_2$ two non-maximal and non-minimal elements of $<$ such that $s_1>s_2$. By our proof of step 1, we can assume $s_1$ and $s_2$ to be two non-trivial hyperbolic sets. Apply an attracting DA map to $s_1$ and a repelling DA map to $s_2$. A new  attracting point has appeared next to $s_1$ and a new repelling point next to $s_2$. Remove a small neighbourhood of each of those points and in this way we obtain a hyperbolic plug (defined in section 3) with one useful entrance boundary and one useful exit boundary. The entry fundamental annulus of this plug consists of two free stable separatrices issued from $s_2$ (see figure \ref{DA map}) and stable manifolds of attracting fixed points. Analogously, its exit fundamental annulus consists of two free unstable separatrices issued from $s_1$ and unstable manifolds of repelling fixed points. By gluing the exit fundamental annulus to the entry fundamental annulus, in a way such that the free separatrices of $s_1$ intersect the free separatrices of $s_2$ transversally, we obtain a new Smale diffeomorphism $\tilde{f}$ acting on a new surface. The Smale order $<_2$ associated to $\tilde{f}$ is the same as $<_1$, only that one relation has been added: $s_1>_2 s_2$.}
\par{We can now repeat this process for every two saddle-type pieces $s$ and $s'$ such that $s>s'$. The only difference now is that since our diffeomorphism $\tilde{f}$ no longer satisfies the hypothesis of Step 1, then the new plugs we'll construct after the DA perturbation, will have more complicated entry and exit laminations. However, a transverse gluing such that the free stable separatrices of $s'$ intersect transversally those of $s$ is still possible by step 3 of our proof of theorem \ref{Smalequestion}. At the end of this process, we get a realisation of $(F,<)$ as the Smale order of a Smale diffeomorphism acting on a closed surface.}
\end{proof}
\section{Variations of Smale's question}
\subsection{What happens if we fix the genus?}
Here's an interesting question generalising the theorem \ref{Stablecase}: Consider $g\in \mathbb{N}$. Which partial orders on finite sets satisfying the connectivity condition can be realised as Smale orders of a Smale diffeomorphism acting on a $g$-torus? By applying the DA map and adding an additional handle between two related basic  pieces, as it was done in the final paragraph of the proof of theorem \ref{Stablecase}, we can increase the genus of the surface on which the Smale order has been realised. Therefore, the previous question boils down to the following one: Given a partial order satisfying the connectivity condition, what is the minimal genus $g$ for which the order can be realised as a Smale order of a Smale diffeomorphism acting on the $g$-torus? 
\par{Let us restrict ourselves to orders for which there is no relation between two saddles and look at our proof of step 1 of theorem \ref{Stablecase}. We first chose cycles around every attracting and repelling point satisfying (\ref{equilibrium}). We then glue bands together in order to partition the boundaries of every domain to a finite number of boundary cycles. At the end, we correct the lengths of the boundary cycles in order for them to be realisable by a domain of a hyperbolic piece (lemma \ref{creation}). Every one of the previous steps, can change the genus of the final surface. Therefore, an efficient answer to the question of the minimal genus for this first case relies on the following problems:} 
\begin{enumerate}
\item (Strengthening lemma \ref{creation}) Given $l_1,l_2,...,l_s \in 2\mathbb{N}^{*}$, does there exist a domain of a basic piece with $s$ boundary components with respective lengths $l_1,l_2,...,l_s$? If yes, what is its minimal genus?
\item Can we find an efficient algorithm producing a choice of cycles of bands around all attracting and repelling points and a choice of boundary cycles giving rise to a band gluing as close as possible to a planar gluing (i.e. the graph whose vertices are attracting and repelling points and whose edges are glued bands is embeddable in a surface of low genus)? 
\end{enumerate}

\subsection{The case of non-trivial attractors}

\par{Why is the case with non-trivial attractors different from what we've done until now? Let us further examine this question. Consider a partial order $(F,<)$ that doesn't satisfy the connectivity condition, where no two saddles are related. Suppose that there exists a maximal element $A$, for which the connectivity condition fails. By theorem \ref{Stablecase}, that means that if the order $<$ is realisable, $A$ corresponds necessarily to a non-trivial repeller. The following theorem is proven in \cite{Bonatti}}
\begin{theorem}[Theorem 2.3.4 of \cite{Bonatti}] \label{nontrivial}
If $f$ is a stable diffeomorphism acting on a closed surface and $A$ a non-trivial hyperbolic repeller for $f$ then if there exist $B,C$ hyperbolic pieces such that for the Smale order of $f$ we have $C<B<A$, then $C$ is a periodic attracting point and $B$ a periodic saddle point. 
\end{theorem} 
\begin{corollary}
The orders represented in figure \ref{fig} are not realisable as Smale orders of any stable diffeomorphism.

\begin{figure}[h!]
\centering{}
\includegraphics[scale=0.7]{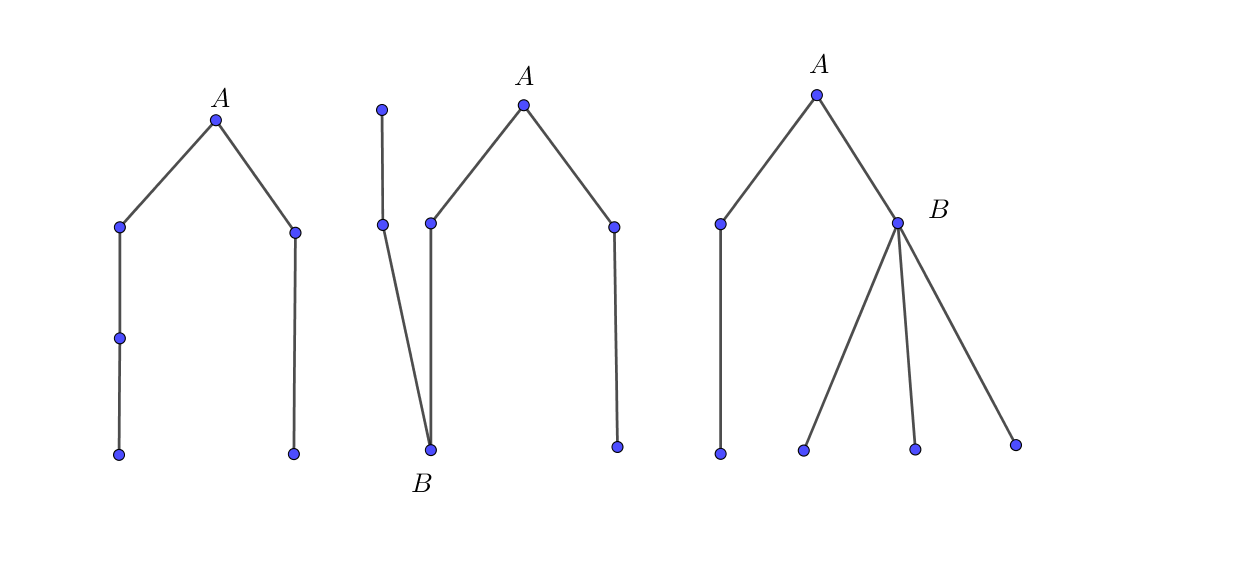}
  \caption{}
  \label{fig}
\end{figure}

\end{corollary}
\begin{proof}
Consider the leftmost order. The point $A$ doesn't verify the connectivity condition, therefore if the order is realisable, it will be represented by a non-trivial repeller. But this contradicts directly, theorem \ref{nontrivial}. Concerning the middle order, by the same argument $A$ must be a non-trivial repeller and by theorem \ref{nontrivial}, $B$ should be an attracting periodic point. But, $B$ doesn't satisfy the connectivity condition either so it should be a non-trivial attractor, which is absurd. Finally, concerning the rightmost order, again $A$ is a non-trivial repeller and by theorem \ref{nontrivial} $B$ is a periodic point. But periodic points have two unstable separatrices, so $B$ cannot be related to three distinct attractors. 
\end{proof}
\par{Let us return to our initial question. If $(F,<)$ is realisable, no saddles in $(F,<)$ are related and there exists a maximal element $A$ that doesn't satisfy the connectivity condition, then by the above theorem all the elements in $F$ smaller than $A$ correspond to periodic points. That means two things:
\begin{enumerate}
\item The non-minimal elements smaller than $A$ have two stable separatrices and two unstable separatrices. Therefore, they're related to at most two minimal elements and two maximal elements.
\item The minimal elements smaller than $A$ must verify the connectivity condition.
\end{enumerate}
The conditions 1, 2 give us additional topological restrictions that we need to satisfy, contrary to the case of trivial attractors and repellers. If a periodic saddle point $p$ arises by condition 1, then all its smaller elements must satisfy the connectivity condition by condition 2. If $p$ is bigger only than one element, then the cycle around this minimal element contains twice the domain of $p$. If not, there exist two minimal elements whose cycles contain just once the domain of $p$. That is the major difference with the case of Smale diffeomorphisms and this is also the reason why it's not always possible to construct cycles satisfying (\ref{equilibrium}) around every attracting and repelling point. To convince you of this, we'll construct a non-realisable order satisfying the above conditions 1, 2.}
\begin{proposition} \label{counterexample} A partial order on a finite set for which no two non-minimal and non-maximal elements are related and which satisfies the conditions 1, 2 isn't necessarily realisable as a Smale order of a stable diffeomorphism.
\end{proposition}
\begin{proof}
\par{Let us represent for every minimal point in $F$ the set of all its bigger elements by a graph, whose vertices are non-maximal points and whose edges are maximal points. Two vertices $s_1,s_2$ are connected by an edge $a$ if $a>s_1,s_2$. This graph is in the general case a graph with multiple edges, loops and edges that connect more than 2 vertices. We can define the same graph for maximal elements. We'll denote on these graphs a maximal or minimal element of $F$ that doesn't satisfy the connectivity condition as an edge with a circle symbol on it and periodic saddle points as red points. Consider the set of graphs defined in figure \ref{graphs}. 
\begin{figure}[h!]
\begin{center}
\includegraphics[scale=0.4]{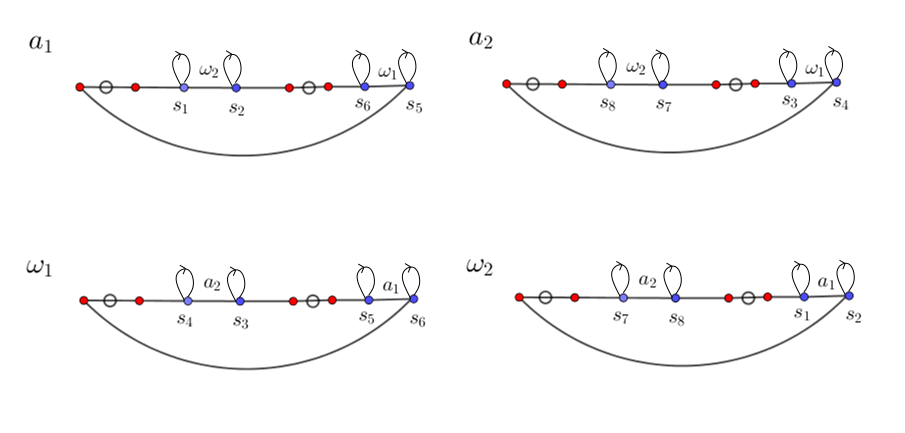}
\caption{}
\label{graphs}
\end{center}
\end{figure}
}
\par{Note that the vertices connected to an edge marked with a circle are all periodic points, which is necessary by theorem \ref{nontrivial}. We'll assume that in the above graphs all the edges and vertices that don't have a name are distinct and disjoint from the set $\lbrace\omega_1,\omega_2,a_1,a_2,s_1,s_2,...,s_8\rbrace$. We'll also assume that all periodic points (red points) are related to exactly two minimal and two maximal elements. It is not difficult to construct a partial order $(F,<)$ for which the saddles are only related to maximal and minimal elements, that satisfies the conditions 1 and 2, such that the graphs associated to the elements $a_1,a_2,\omega_1,\omega_2$ are exactly the graphs drawn in figure \ref{graphs}. Such an order cannot be realised as a Smale order. }
\par{Indeed, choosing a cycle (defined in section 5) around the repelling point $a_1$ is equivalent to choosing a cycle in its graph using only once every red vertex. This is because by hypotheses all periodic points are related to two minimal and two maximal elements, therefore only one separatrix of any periodic point can intersect the basin of $a_1$. Suppose that we choose a cycle in the graph of $a_1$ that visits the red points in the clockwise direction. This implies that $\#(s^{a_1}_1 \xrightarrow{\omega_2} s^{a_1}_2) = \#(s^{a_1}_2 \xrightarrow{\omega_2} s^{a_1}_1) + 1$. Remember now that in the previous section, a necessary condition for the realisability of an order, that made it possible for us to glue bands together, was that for every $i,j,k,l$ we have $\#(s^{\omega_i}_l \xrightarrow{a_j} s^{\omega_i}_k) = \#(s^{a_j}_l \xrightarrow{\omega_i} s^{a_j}_k)$. In order for us to ensure this, we need to choose the cycle around $\omega_2$ accordingly and more specifically, we need to choose a cycle visiting the red points in the clockwise direction. For the same reason, we need to choose a cycle around $a_2$ visiting the red points in the counter-clockwise direction, a cycle around $\omega_1$ that visits red points in the clockwise direction and finally a cycle around $a_1$ that visits red points in the counter-clockwise direction. This is impossible from our initial hypothesis. Therefore, any order $(F,<)$ satisfying the above properties cannot be realised as a Smale order.}
\end{proof} 
\subsection{The case of Morse-Smale diffeomorphisms}
The case of Morse-Smale diffeomorphisms is of course the case with the most topological restrictions: every basic piece is a periodic point. The example described in proposition \ref{counterexample} can be adapted in the Morse-Smale case. Therefore, in order for a partial order on a finite set $(F,<)$ to be realised by a Morse-Smale diffeomorphism the order must satisfy the connectivity condition, the non-maximal and non-minimal elements of $F$ must have at most two stable and two unstable separatrices and finally one should also take measures to avoid examples like the one we've proposed in proposition \ref{counterexample}.
\par{Let us consider the simplest case scenario: the case of gradient-like diffeomorphisms with fixed points. The problem of verifying whether a partial order can be realised by a gradient-like diffeomorphism, is actually closely related to a beautiful problem of simultaneous graph embedding. Consider a partial order $(F,<)$ whose associated graph (defined in the proof of \ref{Smalequestion}) has only first generation saddles with two stable and two unstable separatrices, attractors and repellers. This kind of order is not necessarily realisable by a gradient-like diffeomorphism, but it is not difficult to show the following:}
\begin{theorem}
Any embedded graph in a surface whose complement consists of a finite union of open disks determines uniquely a gradient flow on this surface. The vertices of the graph will correspond to repelling points, each edge to the union of a saddle with its two stable separatrices and finally a disk of its complement will correspond to a basin of attraction of a sink. 
\end{theorem} 
We'll call the \emph{graph of the two highest levels} of $(F,<)$ the graph whose vertices are the maximal elements of $F$ and whose edges correspond to the first generation saddles of $F$. We draw an edge between two distinct vertices $a,b$ if there exists a saddle $s<a,b$. We also draw a loop around a vertex $a$ if there exists a saddle $s$ such that $a$ is the unique element in $F$ bigger than $s$. In the same way we can define the \emph{graph of the two lowest levels} of $(F,<)$. Every graph with loops is embeddable in a surface in a way such that its complement consists of a finite union of disks. We'll call this type of embedding a \emph{good embedding}. In this way, we prove the two following corollaries:
\begin{corollary}
There exists a surface $S_g$ and a gradient-like diffeomorphism with fixed points acting on $S_g$, whose Smale order $(G,<_1)$ contains the exact same relations between saddles and sources as in $(F,<)$. In other words, the graphs of the two highest levels of $(G,<_1)$ and $(F,<)$ are isomorphic. 
\end{corollary}
\begin{corollary}
The partial order $(F,<)$ is realisable by a gradient-like diffeomorphism with fixed points if and only if the graphs of its two highest levels and its two lowest levels are embeddable in the same surface in a way such that the image of one is the dual graph of the other.
\end{corollary}
We now can understand that verifying whether a partial order on a finite set is realisable by a gradient-like diffeomorphism with fixed points is actually related to a non-trivial graph problem. Finding a good embedding for the graph of the highest two levels is a rather easy task, but the number of its non topologically equivalent embeddings is so big, that we haven't been able to think of an efficient algorithm checking the necessary and sufficient condition of the previous corollary.  
\subsection{Open questions}
\par{To conclude, let us mention a series of open questions that arise from our previous discussions:
\begin{enumerate}
\item Given a partial order satisfying the connectivity condition. What is the minimal $g\in \mathbb{N}$ for which the order can be realised as a Smale order of a Smale diffeomorphism acting on the $g$-torus? Can we find a lower bound for it?
\item Does there exist an algorithm of sub-exponential time complexity verifying whether a partial order is realisable by a gradient-like diffeomorphism in dimension 2?
\item Which are the Smale orders associated to Morse-Smale diffeomorphisms acting on a closed 3-manifold?
\item Are all Smale orders associated to a structurally stable diffeomorphism acting on a closed 3-manifold? 
\end{enumerate}   

\end{document}